\documentclass{article}
\usepackage[utf8]{inputenc} 
\usepackage[T1]{fontenc}    
\usepackage{booktabs}       
\usepackage{nicefrac}       
\usepackage{microtype}      

\usepackage[round]{natbib}
\usepackage{amssymb,amsmath,amsthm,bbm}
\usepackage[margin=1in]{geometry}
\usepackage{verbatim,float,url,dsfont}
\usepackage{graphicx}
\usepackage{algorithm,algorithmic}
\usepackage{mathtools}
\usepackage[colorlinks=true,citecolor=blue,urlcolor=blue,linkcolor=blue]{hyperref}
\usepackage{multirow}

\newtheorem{theorem}{Theorem}
\newtheorem{lemma}{Lemma}
\newtheorem{corollary}{Corollary}
\newtheorem{proposition}{Proposition}
\theoremstyle{definition}
\newtheorem{remark}{Remark}

\newtheorem*{assumption*}{\assumptionnumber}
\providecommand{\assumptionnumber}{}


\makeatletter
\newcommand*\rel@kern[1]{\kern#1\dimexpr\macc@kerna}
\newcommand*\widebar[1]{%
  \begingroup
  \def\mathaccent##1##2{%
    \rel@kern{0.8}%
    \overline{\rel@kern{-0.8}\macc@nucleus\rel@kern{0.2}}%
    \rel@kern{-0.2}%
  }%
  \macc@depth\@ne
  \let\math@bgroup\@empty \let\math@egroup\macc@set@skewchar
  \mathsurround\z@ \frozen@everymath{\mathgroup\macc@group\relax}%
  \macc@set@skewchar\relax
  \let\mathaccentV\macc@nested@a
  \macc@nested@a\relax111{#1}%
  \endgroup
}
\makeatother

\def\R{\mathbb{R}}

\def\E{\mathbb{E}}

\def\T{\mathsf{T}}
\def\Cov{\mathrm{Cov}}
\def\Var{\mathrm{Var}}

\def\tr{\mathrm{tr}}
\def\df{\mathrm{df}}

\def\sign{\mathrm{sign}}

\def\hbeta{\hat{\beta}}
\def\htheta{\hat{\theta}}

\def\Risk{\mathrm{Risk}}
\def\Err{\mathrm{Err}}
\def\SURE{\mathrm{SURE}}
\def\hCov{\widehat\Cov}

\def\hdf{\widehat\df}
\def\Ef{\mathrm{Efr}}
\def\BY{\mathrm{BY}}
\def\CB{\mathrm{CB}}
\def\pb{{*b}}
\def\mb{{\dagger{b}}}
\def\Bias{\mathrm{Bias}}
\def\IVar{\mathrm{IVar}}
\def\RVar{\mathrm{RVar}}
\def\Cor{\mathrm{Cor}}

\usepackage[shortlabels]{enumitem}
\usepackage{subcaption}

\title{Unbiased Risk Estimation in the Normal Means Problem via Coupled 
  Bootstrap Techniques}
\author{Natalia L.\ Oliveira$^{1,2}$ \and Jing Lei$^{1}$ \and Ryan J.\ Tibshirani$^{1,2}$} 
\date{$^1$Department of Statistics and Data Science, 
  $^2$Machine Learning Department\\ 
    Carnegie Mellon University}

\begin{document}
\maketitle

\begin{abstract} 
We develop a new approach for estimating the risk of an arbitrary estimator of
the mean vector in the classical normal means problem. The key idea is to
generate two auxiliary data vectors, by adding carefully constructed normal
noise vectors to the original data.  We then train the estimator of interest on
the first auxiliary vector and test it on the second.  In order to stabilize the 
risk estimate, we average this procedure over multiple draws of the synthetic
noise vector. A key aspect of this \emph{coupled bootstrap} (CB) approach is
that it delivers an unbiased estimate of risk under no assumptions on the
estimator of the mean vector, albeit for a modified and slightly ``harder''
version of the original problem, where the noise variance is elevated. We prove 
that, under the assumptions required for the validity of Stein’s unbiased risk 
estimator (SURE), a limiting version of the CB estimator recovers SURE
exactly. We then analyze a bias-variance decomposition of the error of the CB
estimator, which elucidates the effects of the variance of the auxiliary noise 
and the number of bootstrap samples on the accuracy of the estimator. Lastly, we  
demonstrate that the CB estimator performs favorably in various simulated
experiments.
\end{abstract}

\section{Introduction}
\label{sec:introduction}

Given a model that has been fitted on a particular data set, assessing its
risk---typically defined in terms of the accuracy in estimating some population
parameter, or its prediction error---typically defined in terms of the accuracy
in predicting new (unseen) observations, are fundamental questions in both 
classical statistical decision theory and modern statistical machine learning.
Estimates of risk or prediction error can be used for a multitude of purposes, 
e.g., serving as a key input for a decision point (\emph{is the given model good 
  enough to be deployed?}), or a tool for model selection (\emph{is model A  
  preferred over model B?}), or model tuning (\emph{which level of
  regularization strength should be used?}). Naturally, methodology for
risk and prediction error estimation has received considerable attention over
the years in the literature, with foundational work contributed by Akaike,
Mallows, Stein, Efron, Breiman, and others (precise references and discussion
to be given shortly). The current paper revisits this classical topic and
proposes a method to estimate the risk---or equivalently the prediction error in
a fixed-X regression model---based on an auxiliary randomization scheme that
avoids data splitting or resampling techniques.

To fix notation, consider a standard normal means setting, where we observe 
data $Y = (Y_1,\ldots,Y_n) \in \R^n$ distributed according to:     
\begin{equation}
\label{eq:data_model}
Y \sim N(\theta, \sigma^2 I_n),
\end{equation}
where $\theta \in \R^n$ is an unknown parameter to be estimated.  The marginal  
error variance $\sigma^2 > 0$ is assumed to be known, and $I_n$ denotes the $n
\times n$ identity matrix.  An estimator in the context of this problem is
simply a measurable function $g : \R^n \to \R^n$ that, from $Y$, produces an
estimate \smash{$\htheta = g(Y)$} of the mean vector $\theta \in \R^n$.   Given
a loss function $L : \R^n \times \R^n \to \R$, the risk of $g$ is defined by its 
expected loss to $\theta$,
\begin{equation}
\label{eq:general_risk}
\Risk(g) = \E[L(\theta, g(Y))].
\end{equation}
In what follows, without further specification, we work under quadratic loss, so
that the above becomes:
\begin{equation}
\label{eq:quadratic_risk}
\Risk(g) = \E\|\theta - g(Y)\|_2^2 = \E\bigg[ \sum_{i=1}^n (\theta_i -  
g_i(Y))^2 \bigg] , 
\end{equation}
with $g_i$ denoting the $i$th component function of $g$.  In the discussion, we
return to a more general setting and consider \eqref{eq:general_risk} in the
case of loss functions defined by a Bregman divergence.  

In this work, we propose and study a method that we call the \emph{coupled
  bootstrap} to estimate the risk of an arbitrary function $g$ as in
\eqref{eq:quadratic_risk}. Before introducing this method, we discuss the 
connection to prediction error, and introduce and descibe related methods from
the literature, to better contextualize our contributions.   

\subsection{Prediction error, fixed-X regression}

Under quadratic loss (with known $\sigma^2$), estimating risk in
\eqref{eq:quadratic_risk} is equivalent to estimating prediction error, the
expected loss between $g(Y)$ and an independent copy of $Y$, as these two 
quantities are related via   
\begin{equation}
\label{eq:pred_error}
\E\|\tilde{Y} - g(Y)\|_2^2 = \E\|\theta - g(Y)\|_2^2 + n\sigma^2, \quad  
\text{where $\tilde{Y} \sim N(\theta, \sigma^2 I_n)$, independent of $Y$}. 
\end{equation}
In other words, prediction error and risk differ only by a known constant
$n\sigma^2$. In the literature, the choice of whether to focus on prediction
error or risk is generally based on which is the more natural metric in the
applications they use to motivate the study at hand.

An important special case of the normal means problem in which prediction error
is a common focus is \emph{fixed-X regression}: here $Y \in \R^n$ is viewed as a
response vector that is paired with a feature matrix $X \in \R^{n \times p}$
i.e., the $i$th row of $X$ is a feature vector associated with $Y_i$, and $g$
usually performs a kind of regression of $Y$ on $X$.  For example, when
\smash{$g(Y) = X \hbeta$} for some estimated coefficient vector \smash{$\hbeta
  \in \R^p$}, and the mean in \eqref{eq:data_model} is itself linear in $X$,
i.e., $\theta = X \beta$ for some coefficient vector $\beta \in \R^p$, then the
decomposition in \eqref{eq:pred_error} becomes
\begin{equation}
\label{eq:pred_error_x}
\E\|\tilde{Y} - X \hbeta\|_2^2 = \E\|X \beta - X \hbeta\|_2^2 + n\sigma^2, \quad   
\text{where $\tilde{Y} \sim N(X \beta, \sigma^2 I_n)$, independent of $Y$}. 
\end{equation}
We emphasize that we are treating $X$ here as \emph{fixed} (nonrandom);
further, by measuring prediction error as in \eqref{eq:pred_error_x}, we are
treating $X$ as the \emph{common} set of features that are used across both 
training and testing (i.e., \smash{$\tilde{Y}$} is a new vector of response
values, but observed at the same features as $Y$).    

Much of the classical literature on prediction error estimation in statistics
falls in the fixed-X regression setting, with, e.g., \emph{Mallow's Cp}
\citep{mallows1973comments} marking a seminal early contribution in this area.
In some applications of regression (such as experimental design), the fixed-X 
perspective is natural; yet in others, a \emph{random-X} perspective is more
natural, where prediction error is measured with respect to a \emph{new}
feature vector (drawn i.i.d.\ from the same distribution as the training
features). It is worth being clear that prediction error in the fixed-X and
random-X sense are generally \emph{not} equivalent and admit critical
differences (see, e.g., \citet{rosset2020from} for a discussion); thus,
methods for estimating prediction error in the fixed-X setting do not
necessarily translate to the random-X setting, and vice versa.  For example,
sample splitting and cross-validation are arguably the most widely-used tools
for estimating random-X prediction error, but are not generally applicable for
fixed-X (when $X$ is fixed, there is no reason to believe that a random subset
of its rows will be representative of the full set). On the other hand, the CB
estimator that we will develop is aligned with fixed-X prediction error, but not 
random-X prediction error in general.  

To summarize, in this paper, we choose to focus on risk as in
\eqref{eq:quadratic_risk} for simplicity of exposition,
but as we explained above, our results translate over to prediction error in
\eqref{eq:pred_error}, which encompasses fixed-X regression error as in
\eqref{eq:pred_error_x}. In what follows, we will move back and forth between  
the two concepts (risk and prediction error) fluidly, as needed.     

\subsection{Stein's unbiased risk estimator}
\label{sec:sure}

One of the most well-known and widely-used risk estimators in the normal 
means problem is due to \citet{stein1981estimation}. For concreteness, we
translate this result into the notation of our paper. 

\begin{theorem}[\citealt{stein1981estimation}]
\label{thm:sure}
Let $Y \sim N(\theta, \sigma^2 I_n)$. Let $g : \R^n \to \R^n$ be weakly
differentiable\footnote{Weak differentiability of $g$ is actually a slightly
  stronger assumption than needed, but is stated for simplicity; see
  Remark \ref{rem:weak_diff}.}, and write $\nabla_i g_j$ for the weak partial
derivative of component function $g_j$ with respect to variable $y_i$.  Assume
that $\E\|g(Y)\|_2^2 < \infty$, and $\E|\nabla_i g_i(Y)| < \infty$, for 
$i=1,\ldots,n$. Denote the divergence of $g$ by \smash{$\nabla \cdot g =
  \sum_{i=1}^n \nabla_i g_i$}, and define      
\begin{equation}
\label{eq:sure}
\SURE(g) = \|Y - g(Y)\|_2^2 + 2\sigma^2 (\nabla \cdot g)(Y) - n\sigma^2,
\end{equation}
Then the above provides an unbiased estimator of risk: $\E[\SURE(g)] =
\Risk(g)$.      
\end{theorem}

The estimator defined in \eqref{eq:sure} is known as Stein's unbiased risk
estimator (SURE). Ignoring the last term: $-n\sigma^2$, a constant not
depending on $g$, the first two terms here are the observed training error:
$\|Y - g(Y)\|_2^2$, and a measure of complexity: $2\sigma^2 (\nabla \cdot
g)(Y)$.  At the heart of Theorem \ref{thm:sure} is a result known as
\emph{Stein's formula}, which says for weakly differentiable $g$ 
\citep{stein1981estimation},
\begin{equation}
\label{eq:stein_formula}
\frac{1}{\sigma^2} \Cov(Y_i, g_i(Y)) = \E[ \nabla_i g_i(Y) ], \quad
i=1,\ldots,n. 
\end{equation}
Recall that the \emph{(effective) degrees of freedom} of $g$ is defined by
\citep{hastie1990generalized, ye1998measuring}: 
\begin{equation}
\label{eq:df}
\df(g) = \frac{1}{\sigma^2} \sum_{i=1}^n \Cov(Y_i, g_i(Y)). 
\end{equation}
This measures complexity based on the association (summed over the training
set) between each $Y_i$ and the corresponding estimate $g_i(Y)$ of $\theta_i$
(generally speaking, the more complex $g$ is, the greater this association will
be).  Note that, according to \eqref{eq:stein_formula}, \eqref{eq:df}, the
second term in \eqref{eq:sure} leverages an unbiased estimator for degrees of
freedom: $\E[(\nabla \cdot g)(Y)] = \df(g)$.      

\subsection{Efron, Breiman, and Ye}
\label{sec:eby}

For arbitrary $g$, we can always decompose its risk by:
\begin{equation}
\label{eq:cov_decomp}
\Risk(g) = \E\|Y - g(Y)\|_2^2 + 2 \sum_{i=1}^n \Cov(Y_i, g_i(Y)) - n\sigma^2,    
\end{equation}
which follows from simple algebra (add and subtract $Y$ inside the expectation
in $\E\|\theta - g(Y)\|_2^2$, and expand the quadratic). This is often referred
to as \emph{Efron's covariance decomposition} (or \emph{Efron's optimism
  theorem}), after \citet{efron1975defining, efron1986biased,
  efron2004estimation}. We reiterate that the covariance decomposition in
\eqref{eq:cov_decomp} holds for any function $g$.  The same is true of the
definition of degrees of freedom in \eqref{eq:df}: it applies to any $g$. In
fact, these do not even require normality of the data vector: \eqref{eq:df},
\eqref{eq:cov_decomp} only require the distribution of $Y$ to be isotropic
(i.e., to have a covariance matrix $\sigma^2 I_n$).  Meanwhile, Stein's formula
\eqref{eq:stein_formula}, and hence the unbiasedness of SURE \eqref{eq:sure},
only holds for a weakly differentiable $g$, and Gaussian $Y$.

Efron's covariance decomposition reveals that, to get an unbiased estimator of
$\Risk(g)$, we only need an unbiased estimator of the second term: \smash{$2
  \sum_{i=1}^n \Cov(Y_i, g_i(Y))$}, called the \emph{optimism} of $g$.  This is
because the first term, the (expected) training error, clearly yields the
observed training error as its unbiased estimator.  A natural way to estimate 
optimism is to use the bootstrap, or more precisely, the \emph{parametric
  bootstrap}. This has been pursued by several authors, notably
\citet{breiman1992little, ye1998measuring, efron2004estimation}. In the
parametric bootstrap, we generate samples  
\begin{equation}
\label{eq:bootstrap}
Y^{*b} \,|\, Y \sim N(Y, \alpha \sigma^2 I_n), \quad 
\text{independently}, \quad \text{for $b=1,\ldots,B$},
\end{equation}
for some constant $\alpha > 0$ (typically $\alpha \leq 1$).  We then form the
estimates:  
\begin{equation}
\label{eq:bootstrap_cov}
\hCov_i^* = \frac{1}{B-1} \sum_{b=1}^B (Y_i^{*b} - \bar{Y}_i^*) g_i(Y^{*b}),
\quad i=1,\ldots,n,       
\end{equation}
where \smash{$\bar{Y}_i^* = \frac{1}{B} \sum_{b=1}^B Y^{*b}_i$}, $i=1,\ldots,n$ 
are the bootstrap means of the coordinates.  \citet{efron2004estimation} also  
presents a more general framework in which, instead of \eqref{eq:bootstrap}, we
draw bootstrap samples from \smash{$N(\check\theta, \alpha \sigma^2 I_n)$}, for
some seed estimate \smash{$\check\theta$}.  As an effort to reduce bias, Efron
recommends using a more flexible model for estimating \smash{$\check\theta$} 
compared to that for \smash{$\htheta = g(Y)$}, where the ``ultimate'' flexible
model (as Efron calls it) reduces to \smash{$\check\theta = Y$}, in
\eqref{eq:bootstrap}.  This is also the choice made in both
\citet{breiman1992little} and \citet{ye1998measuring}.       

While there are strong commonalities among the parametric bootstrap proposals of
Efron, Breiman, and Ye, all three being centered around
\eqref{eq:bootstrap_cov}, there are also noteworthy differences in how these  
authors use \eqref{eq:bootstrap_cov} in order to estimate risk.  Efron proposes
the risk estimator:  
\begin{equation}
\label{eq:efron_risk}
\Ef_\alpha(g) = \|Y - g(Y)\|_2^2 + 2 \sum_{i=1}^n \hCov_i^* - n\sigma^2,   
\end{equation}
whereas Breiman and Ye effectively propose the risk estimator:  
\begin{equation}
\label{eq:by_risk}
\BY_\alpha(g) = \|Y - g(Y)\|_2^2 + \frac{2}{\alpha} \sum_{i=1}^n \hCov_i^* - n
\sigma^2. 
\end{equation}
We say ``effectively'' here because Breiman and Ye consider a slightly 
different estimator than that in \eqref{eq:by_risk}.  See
Appendix \ref{app:by_estimators} for details.  But for a large number of
bootstrap draws $B$, the proposals of Breiman and Ye will behave very similarly
to \eqref{eq:by_risk}, and thus we refer to \eqref{eq:by_risk} as the Breiman-Ye
(BY) risk estimator.     

The difference between \eqref{eq:efron_risk} and \eqref{eq:by_risk} is that in
the latter the sum of estimated covariances is scaled by $1/\alpha$.  Efron,
Breiman, and Ye each generally advocate for choices of $\alpha$ in between 0.6
and 1. For such a large value of $\alpha$, the scaling factor $1/\alpha$ in
\eqref{eq:by_risk} will not play a huge role.  But for small values of
$\alpha$---a regime that is of interest in the current paper---this scaling 
factor will make all the difference.      

\subsection{What are these bootstrap methods estimating?}

The bootstrap methods in \eqref{eq:efron_risk} and \eqref{eq:by_risk}
are well-known and widely-used for estimating risk in normal means
problems. Both are fairly natural. Efron's estimator \eqref{eq:efron_risk} 
directly uses the parametric bootstrap to estimate optimism: \smash{$2
  \sum_{i=1}^n \Cov(Y_i, g_i(Y))$}. For the BY estimator \eqref{eq:by_risk},
writing    
$$
\frac{2}{\alpha} \sum_{i=1}^n \hCov_i^* = 2\sigma^2 
\underbrace{\frac{1}{\alpha \sigma^2} \sum_{i=1}^n \hCov_i^*}_{\hdf(g)},
$$
we see that it can be motivated from the perspective of estimating degrees of
freedom (rather than optimism) via the parametric bootstrap, since the
conditional variance of the bootstrap draws (given $Y$) is $\alpha \sigma^2$.

Now we come to a key point: the motivation given for the above estimators is
based on the \emph{conditional} distribution of bootstrap samples (conditional
on the data $Y$). However, their performance as risk estimators hinges on how
they behave \emph{marginally} over $Y$, and unfortunately, from the marginal
point of view, it is not as clear what these methods are actually targeting. We 
discuss this for each method separately.        

\subsubsection{Efron's Estimator}
\label{sec:efron}

First, consider Efron's estimator in \eqref{eq:efron_risk}.  Write $Y^*$ 
for a single bootstrap draw, i.e., $Y^*\,|\, Y \sim N(Y, \alpha \sigma^2 
I_n)$. As this estimator treats $Y^*$ as the data vector (in place of $Y$), one
might suppose that marginally it targets the optimism of $g$, but at an elevated
noise level $(1+\alpha) \sigma^2$ (instead of $\sigma^2$), because $Y^*\sim
N(\theta, (1+\alpha) \sigma^2)$. However, its expectation does not really
support this claim.  To see this, first observe that
$$
\E\big[ \hCov_i^* \,\big|\, Y \big] = \Cov\big(Y_i^*, g_i(Y^*) \,\big|\, 
Y\big).
$$
Here we simply used the fact that an empirical covariance computed from i.i.d.\
samples of a pair of random variables is unbiased for their covariance
(everything here being conditional on $Y$).  Next observe that
\begin{equation}
\label{eq:law_total_cov}
\sum_{i=1}^n \Cov(Y_i^*, g_i(Y^*)) = \underbrace{\sum_{i=1}^n   
  \E\big[\Cov\big(Y_i^*, g_i(Y^*) \,\big|\, Y\big)\big]}_{A_\alpha} \,+\,    
\underbrace{\sum_{i=1}^n \Cov(Y_i, g_i(Y^*))}_{B_\alpha},
\end{equation}
by the law of total covariance, and where we used $\Cov(\E[Y_i^* \,|\, Y],
\E[g_i(Y^*) \,|\, Y]) = \Cov(Y_i, g_i(Y^*))$, for each summand in the second
term, which follows from a short calculation. Therefore Efron's method delivers
a covariance term with marginal expectation:       
\begin{equation}
\label{eq:efron_cov_mean}
\E\bigg[\sum_{i=1}^n \hCov_i^*\bigg] = \E\bigg[\sum_{i=1}^n \Cov\big(Y_i^*,  
g_i(Y^*) \,\big|\, Y\big)\bigg].
\end{equation}
This only captures a part of the optimism of $g$ at the elevated noise level
$(1+\alpha) \sigma^2$, labeled $A_\alpha$ in \eqref{eq:law_total_cov}, and 
not a second part, labeled $B_\alpha$ in \eqref{eq:law_total_cov}.   

Based on this, we can reason that for small $\alpha$, the bootstrap estimator 
\smash{$\sum_{i=1}^n \hCov_i^*$} will typically be badly biased for the
noise-elevated covariance \smash{$\sum_{i=1}^n \Cov(Y_i^*, g_i(Y^*))$}, and
hence also badly biased for the original covariance \smash{$\sum_{i=1}^n
  \Cov(Y_i, g_i(Y))$} (as this will be close to the noise-elevated version).
This is because it will be concentrated around $A_\alpha$ in
\eqref{eq:law_total_cov}, which will typically be small in comparison to the
second component $B_\alpha$ in \eqref{eq:law_total_cov}.  For example, for a
linear smoother $g(Y) = SY$ (for a fixed matrix $S \in \R^{n \times n}$), note
that 
\begin{equation}
\label{eq:ab_linear}
A_\alpha = \alpha \sigma^2 \tr(S) \quad \text{and} \quad 
B_\alpha = \sigma^2 \tr(S),
\end{equation}
and the latter term will dominate for small $\alpha$.  Similar arguments hold 
for locally linear $g$ (well-approximated by its first-order Taylor expansion).  

Meanwhile, for moderate $\alpha$, the estimator \smash{$\sum_{i=1}^n
  \hCov_i^*$} \emph{can} have low bias for \smash{$\sum_{i=1}^n \Cov(Y_i,
  g_i(Y))$}  (this is the original covariance and \emph{not} the noise-elevated
version, which will be generally larger for moderate $\alpha$), if we are able
to choose $\alpha$ such that \smash{$A _\alpha \approx \sum_{i=1}^n \Cov(Y_i,
  g_i(Y))$}. For linear smoothers, as we can see from \eqref{eq:ab_linear}, we 
simply need to take $\alpha=1$.  In general, however, it will not be at all
clear how to choose $\alpha$ appropriately, as it will be unclear how
$A_\alpha$ behaves with $\alpha$.  More broadly, for any given value of
$\alpha$ in hand, it is not clear precisely what is being targeted in
\eqref{eq:efron_cov_mean}, and thus, not clear precisely what risk is being
estimated by \eqref{eq:efron_risk}.

\subsubsection{Breiman-Ye Estimator}

Next, consider the BY estimator in \eqref{eq:by_risk}. By the same calculations
as in the last case, we see that the BY method uses a covariance term with
marginal expectation:   
\begin{equation}
\label{eq:by_cov_mean}
\frac{1}{\alpha} \E\bigg[\sum_{i=1}^n \hCov_i^*\bigg] = \frac{1}{\alpha}
\E\bigg[\sum_{i=1}^n \Cov\big(Y_i^*, g_i(Y^*) \,\big|\, Y\big)\bigg].
\end{equation}
The sum above only captures one part of the optimism at the elevated noise level
$(1+\alpha) \sigma^2$, labeled $A_\alpha$ in \eqref{eq:law_total_cov}, but the
sum is also inflated by division by $\alpha$ (recall, usually $\alpha \leq 1$).
This makes the behavior of the BY method more subtle than that of Efron's
method; we seek $\alpha$ so that \smash{$A_\alpha/\alpha \approx \sum_{i=1}^n 
  \Cov(Y_i, g_i(Y))$}, yet it is unclear whether this means that we should
choose $\alpha$ to be small or large.  

The case of a linear smoother $g(Y) = SY$ is encouraging: recalling
\eqref{eq:ab_linear}, we have $A_\alpha/\alpha = \sigma^2 \tr(S)$, which is
equal to \smash{$\sum_{i=1}^n \Cov(Y_i, g_i(Y))$} for any value of $\alpha$. Of
course, in general we will not be so lucky, and varying $\alpha$ will vary 
$A_\alpha/\alpha$, hence vary what we are targeting in
\eqref{eq:by_cov_mean}. This brings us to the same general difficulty with the BY
estimator as in the last case: for any given choice of $\alpha$, it is unclear
what quantity is actually being estimated by \smash{$\frac{1}{\alpha}
  \sum_{i=1}^n \hCov_i^*$}, and thus, unclear precisely what risk is being 
estimated by \eqref{eq:by_risk}.

\subsection{Proposed estimator}

The main proposal in this paper is a new estimator for the risk of an arbitrary
function $g$, based on bootstrap draws as in \eqref{eq:bootstrap}.  The key
motivation for our estimator is that, for any $\alpha$, it will be unbiased for
an intuitive, explicit target: the risk of $g$ at the noise level of $(1+\alpha)
\sigma^2$, which we denote by     
\begin{equation}
\label{eq:risk_alpha}
\Risk_\alpha(g) = \E\|\theta -  g(Y_\alpha)\|_2^2, 
\quad \text{where $Y_\alpha \sim N(\theta, (1+\alpha) \sigma^2 I_n)$}. 
\end{equation}
One can think of $\Risk_\alpha(g)$ as the risk for a ``harder'' version of the
original problem, where the mean $\theta$ is the same, but the noise variance
$\sigma^2$ is multiplied by a factor of $1+\alpha$. Later (in Proposition 
\ref{prop:risk_alpha_smoothness}), we will show that $\Risk_\alpha(g)$ converges 
to $\Risk(g)$ as $\alpha \to 0$, and in fact, does so smoothly: it is 
continuously differentiable in $\alpha$, under only mild moment conditions on
$g(Y)$.

In order to estimate $\Risk_\alpha(g)$, we take an approach that departs in two
ways from prior work.  First, we do not rely on the covariance decomposition 
\eqref{eq:cov_decomp}, and do not frame the problem in terms of directly
estimating optimism (or degrees of freedom); this circumvents the need to
estimate a covariance with the bootstrap (and as such, avoids challenges due to
the law of total covariance \eqref{eq:law_total_cov}). Second, coupled with each
bootstrap draw in \eqref{eq:bootstrap}, we carefully generate another bootstrap
draw that is \emph{marginally} independent from it (which gives us a total of
$2B$ draws). In particular, we generate samples according to:  
\begin{equation}
\label{eq:cb}
\begin{gathered}
\omega^b \sim N(0, \sigma^2 I_n), \quad 
\text{independently}, \quad \text{for $b=1,\ldots,B$}, \\ 
Y^\pb = Y + \sqrt\alpha \omega^b, \quad 
Y^\mb = Y - \omega^b / \sqrt\alpha, \quad \text{for $b=1,\ldots,B$}, 
\end{gathered}
\end{equation}
for some constant $\alpha > 0$, and based on these samples, we define the risk
estimator:  
\begin{equation}
\label{eq:cb_risk}
\CB_\alpha(g) = \frac{1}{B} \sum_{b=1}^B \Big(\|Y^\mb - g(Y^\pb)\|^2_2 -
\|\omega^b\|_2^2 / \alpha\Big) - n\sigma^2.
\end{equation} 
The intuition here is that each pair \smash{$(Y^\pb, Y^\mb)$} comprises two
independent samples from a normal distribution with mean $\theta$, and hence 
each squared error term \smash{$\|Y^\mb - g(Y^\pb)\|_2^2$} imitates the
prediction error incurred by $g(Y)$ at a new copy of $Y$.  Together, the
remaining terms $-\|\omega^b\|_2^2 / \alpha$ (in each summand) and
$-n\sigma^2$ adjust for the fact that \smash{$Y^\pb$} and \smash{$Y^\mb$} have
different variances, and bring us from the prediction scale to the risk scale
(recall \eqref{eq:pred_error}).  In this paper, we refer to \eqref{eq:cb_risk}
as the \emph{coupled bootstrap} (CB) risk estimator.         

In \eqref{eq:cb_risk}, as $g$ is applied to a noise-elevated draw
\smash{$Y^\pb$} that has mean $\theta$ and variance $(1+\alpha) \sigma^2$, one  
might conjecture that we are targeting risk (or prediction error) at the
noise-elevated level $(1+\alpha) \sigma^2$.  Later, when we provide more details  
behind the construction of the CB estimator \eqref{eq:cb_risk}, we will show (in
Corollary \ref{cor:cb_unbiased}) that this is indeed true: $\E[\CB_\alpha(g)] =
\Risk_\alpha(g)$.  This is a strong property, and it holds without any
assumptions on $g$ whatsoever.         

\subsection{Summary of contributions}

The following is a summary of our main contributions and an outline for this
paper. 

\begin{itemize}
\item In Section \ref{sec:basic_properties}, we examine basic properties of
  the CB risk estimator, which includes proving that for any $g$ and any
  $\alpha$, the CB estimator is unbiased for $\Risk_\alpha(g)$.   

\item In Section \ref{sec:noiseless_limit}, we study the behavior of the
  CB estimator as $B \to \infty$ and $\alpha \to 0$, and prove that under the    
  same smoothness assumptions on $g$ as those in  \citet{stein1981estimation}
  (to guarantee unbiasedness of SURE; recall Theorem \ref{thm:sure}), the
  limiting CB estimator recovers SURE exactly.  

\item In Section \ref{sec:bias_variance}, we analyze the bias and variance
  (quantifying their dependence on $\alpha$ and other problem parameters) of the
  CB estimator when it is viewed as an estimator of $\Risk(g)$, the original
  risk.  Insights from this include a recommendation to choose the number of
  bootstrap draws $B$ to scale with $1/\alpha$, for small $\alpha$, in order to
  control the variance of the CB estimator. 

\item In Section \ref{sec:experiments}, we compare the CB estimator to the
  existing bootstrap methods (Efron and BY) for risk estimation in simulations. 
  We find that the CB estimator generally performs favorably, particularly so
  when $g$ is unstable.    

\item In Section \ref{sec:discussion}, we conclude with a discussion, and give
  an extension of our coupled bootstrap framework to the setting of structured  
  errors (i.e., a non-isotropic covariance in \eqref{eq:data_model}), as well as 
  extensions to other loss functions and distributions.       
\end{itemize}

\subsection{Related work}
\label{sec:related_work}

Risk (or prediction error) estimation is a well-studied topic and has a rich
history in statistics.  What follows is by no means comprehensive, but is a
selective review of papers that are most related to our paper, apart from
\citet{breiman1992little, ye1998measuring, efron2004estimation}, which have
already been discussed in some detail. 

In a sense, covariance penalties originated in the work of
\citet{akaike1973information} and \citet{mallows1973comments}, who focused on
classical likelihood-based models and fixed-X linear regression,
respectively. \citet{stein1981estimation} greatly extended the scope of models
under consideration (or in our notation, functions $g$ whose risk is to be 
estimated) with SURE, which applies broadly to models whose predictions vary
smoothly with respect to the input data $Y$; recall Theorem \ref{thm:sure}.
Stein's work has had a huge impact in both statistics and signal processing, and
SURE is now a central tool in wavelet modeling, image denoising, penalized
regression, low-rank matrix factorization, and other areas; see, e.g.,
\citet{donoho1995adapting, cai1999adaptive, johnstone1999wavelet,
  blu2007surelet, zou2007degrees, zou2008regularized, 
  tibshirani2011solution, tibshirani2012degrees, candes2013unbiased,
  ulfarsson2013tuning1, ulfarsson2013tuning2, wang2013sure,
  krishnan2013selection}.   

A downside of SURE is that it cannot be applied to various models of interest
(e.g., tree-based methods, certain variable selection methods, and so on), as it 
requires $g$ to be weakly differentiable, which is generally violated when
$g$ is discontinuous.  Meanwhile, even when SURE is applicable, it is often
highly nontrivial to (analytically) calculate the Stein divergence
$\nabla \cdot g$; in fact, the key contribution in many of the papers given in
the last set of references is that the authors were able to calculate this
divergence for an interesting class of models (e.g., wavelet thresholding, total
variation denoising, lasso regression, and so on).   

These shortcomings of SURE are well-known.  Extensions of SURE to accommodate  
discontinuities in $g$ were derived in \citet{tibshirani2015degrees, 
  mikkelsen2018degrees}; see also \citet{tibshirani2019excess}.  While useful in
some contexts, these extensions are generally far more complicated (and harder
to compute) than SURE.  On the computational side, \citet{ramani2008monte}
proposed a Monte Carlo method for approximating SURE that only requires
evaluating $g$ (and not its partial derivatives).  This has since become quite
popular in the signal processing community, see, e.g.,
\citet{chatterjee2009clustering, lingala2011accelerated, metzler2016denoising,  
  soltanayev2018training} for applications of this idea and follow-up work. 

As it turns out, the Monte Carlo SURE approach of \citet{ramani2008monte} is
precisely the same as the bootstrap method of \citet{breiman1992little}.  It is
thus also highly related to the work of \citet{ye1998measuring}, and essentially
equivalent to what we call the BY risk estimator in \eqref{eq:by_risk}; recall
the discussion in Section \ref{sec:eby}.  It seems that Ramani et al.\ were
unaware of the past work of Breiman and Ye.  That being the case, their work
provided an important new perspective on this methodology: they show that for
infinite bootstrap samples ($B=\infty$) and with appropriate smoothness
conditions on $g$, Monte Carlo SURE (and thus the BY estimator in
\eqref{eq:by_risk}) converges to SURE in \eqref{eq:sure} as $\alpha \to 0$.
Breiman and Ye, on their part, seemed unaware of this connection, as they both
cautioned against choosing small values of $\alpha$, advocating for choices of
$\alpha$ upwards of 0.5.

 
Finally, we note that the work of \citet{tian2020prediction} inspired us to
pursue the current paper. Tian proposed the coupled bootstrap approach in
\eqref{eq:cb} (albeit with $B=1$) to estimate the fixed-X regression error of
prediction rules that perform feature selection in a working linear model (such
as the lasso). Their focus was different than ours: they study the estimation of
prediction error conditional on a model selection event (such as the event that
the lasso selects a particular active set). They conduct a bias-variance
analysis as a function of the noise inflation parameter $\alpha$, under the
assumption that the true model is itself linear. They also recommend a
diminishing choice of $\alpha$ as the sample size grows. This is all done in
service of an asymptotic analysis which shows that the estimated prediction
error converges to the true one as $n \to \infty$. The idea of using auxiliary
randomization in the literature on inference after model selection was initiated
by \citet{tian2018selective}, and has since been further developed by several
others, e.g., \citet{rasines2023splitting, leiner2024data, neufeld2024data},
some of this literature developed concurrently with our paper, and some after.   

\section{Basic properties}
\label{sec:basic_properties}

In this section, we investigate basic properties of the CB estimator in
\eqref{eq:cb_risk}, beginning with its unbiasedness for the noise-elevated risk
in \eqref{eq:risk_alpha}.      

\subsection{Unbiasedness for noise-elevated target}

The unbiasedness of CB estimator for the appropriate noise-elevated risk stems
from a simple ``three-point'' formula under squared error loss. Here and
subsequently, we use $\langle a, b \rangle = a^\T b$ for vectors $a,b$.            

\begin{proposition}
\label{prop:three_point}
Let $U,V,W \in \R^n$ be independent random vectors.  Then for any $g$,
\begin{equation}
\label{eq:three_point1}
\E \|V - g(U)\|_2^2 - \E \|W - g(U)\|_2^2 = \E\|V\|_2^2 - \E\|W\|_2^2 + 2
\langle \E[g(U)], \E[W] - \E[V] \rangle,
\end{equation}
assuming all expectations exist and are finite. In particular, if $U,V$ are
i.i.d.\ and $\E[U] = \E[W]$, then
\begin{equation}
\label{eq:three_point2}
 \E\|V - g(U)\|_2^2 = \E \|W - g(U)\|_2^2 + \E\|U\|_2^2 - \E\|W\|_2^2. 
\end{equation}
\end{proposition}

\begin{proof}
The first statement \eqref{eq:three_point1} just follows from expanding the 
quadratic terms and using the independence of $U,V,W$.  The second
statement \eqref{eq:three_point2} follows from the first by noting that if 
$U,V$ are i.i.d.\ and $\E[U] = \E[W]$ then $\E[V] = \E[W]$, thus the first term 
on the right-hand side in \eqref{eq:three_point1} is $\E\|U\|_2^2$, and the last
term is zero.    
\end{proof}

The statements in Proposition \ref{prop:three_point} are the result of somewhat
trivial algebraic manipulations.  Nonetheless, they are useful observations: to
recap, the second display \eqref{eq:three_point2} says that given a random
vector $U$, if we can generate another random vector $W$ that is independent of
$U$ and shares the same mean (importantly, we do \emph{not} require it to be
i.i.d.), then we can unbiasedly estimate the predicion error (or risk) of $g$
applied to $U$.  

This is the basis for the CB risk estimator.  By carefully adding and
substracting noise to $Y$, we generate a pair of random vectors
\smash{$(U,W)=(Y^\pb, Y^\mb)$} that are independent of each other and have a 
common mean $\theta$. Then we pivot slightly from the original problem and now
seek to estimate the risk of $g$ when it is applied to $U$, which has marginal 
distribution $N(\theta, (1+\alpha) \sigma^2)$.  For this task, we have a simple
unbiased estimator, following \eqref{eq:three_point2}.  

\begin{corollary}
\label{cor:cb_unbiased}
Let $Y \sim N(\theta, \sigma^2 I_n)$.  Then for any $g$, any $\alpha > 0$, and
any $B \geq 1$, the CB estimator defined by \eqref{eq:cb}, \eqref{eq:cb_risk} is
unbiased for the noise-elevated risk in \eqref{eq:risk_alpha}:
$\E[\CB_\alpha(g)] = \Risk_\alpha(g)$.  
\end{corollary}

\begin{proof}
For each $b$, note that \smash{$Y^\pb, Y^\mb$} are independent since they are
jointly normal and uncorrelated:  
\begin{align*}
\Cov(Y + \sqrt\alpha \omega^b, Y - \omega^b / \sqrt\alpha) 
&= \Cov(Y, Y) + (\sqrt\alpha - 1/\sqrt\alpha) \Cov(Y, \omega^b) - 
\Cov(\omega^b, \omega^b) \\
&= n\sigma^2 + 0 - n\sigma^2 \\
&= 0.
\end{align*}
They also clearly have the same mean, thus we can apply \eqref{eq:three_point2}
with \smash{$U=Y^\pb, V = \tilde{Y}^\pb, W=Y^\mb$}, where
\smash{$\tilde{Y}^\pb$} is an independent copy of \smash{$Y^\pb$}. This shows
that 
\begin{equation}
\label{eq:cb_pred1}
\|Y^\mb - g(Y^\pb)\|^2_2 + \|Y^\pb\|_2^2 - \|Y^\mb\|_2^2 
\end{equation}
is unbiased for \smash{$\E\|\tilde{Y}^\pb - g(Y^\pb)\|_2^2$}. Now observe that
we can replace \smash{$\|Y^\pb\|_2^2 - \|Y^\mb\|_2^2$} in the above display by 
anything with the same expectation, $n\sigma^2(\alpha - 1/\alpha)$, and the
result will still be unbiased for \smash{$\E\|\tilde{Y}^\pb - g(Y^\pb)\|_2^2$}. 
One such option is  
\begin{equation}
\label{eq:cb_pred2}
\|Y^\mb - g(Y^\pb)\|^2_2 + n\sigma^2\alpha - \|\omega^b\|_2^2/\alpha,
\end{equation}
and thus, after subtracting off $n\sigma^2(1+\alpha)$, we learn that
$$
\|Y^\mb - g(Y^\pb)\|^2_2 - \|\omega^b\|_2^2/\alpha - n\sigma^2
$$
is unbiased for $\Risk_\alpha(g)$ in \eqref{eq:risk_alpha}.  The CB estimator in
\eqref{eq:cb_risk}, being an average of such terms over $b=1,\ldots,B$, is
therefore also unbiased for $\Risk_\alpha(g)$.   
\end{proof}
 
\begin{remark}
\label{rem:cb_unbiased}
In Proposition \ref{prop:three_point}, we require that $U,W$ are independent 
so that we can factorize $\E\langle g(U), W \rangle = \langle \E[g(U)], \E[W] 
\rangle$ in \eqref{eq:three_point1} and hence cancel out this term with
$\langle \E[g(U)], \E[V] \rangle$, when $\E[V] = \E[W]$, to achieve
\eqref{eq:three_point2}.  This is the only reason that we require a normal data
model $Y \sim N(\theta, \sigma^2 I_n)$ for the unbiasedness result in Corollary 
\ref{cor:cb_unbiased}; we can construct \smash{$U=Y^\pb,  W=Y^\mb$} to be 
uncorrelated, but it is only under normality that this will imply independence.        

When $g(Y) = SY$ is linear, if $U,W$ are merely uncorrelated then we still get
the desired factorization:   
$$
\E\langle SU, W \rangle = \E\tr(S U W^\T) = \tr(S \E[U W^\T]) = \tr(S \E[U]
\E[W]^\T) = \langle S\E[U], \E[W] \rangle,
$$
so the unbiasedness result in Corollary \ref{cor:cb_unbiased} still holds under
the weaker conditions: $\E[Y] = \theta, \Cov(Y) = \sigma^2 I_n$. As an example
consequence, this means that the CB estimator for ridge regression is still
unbiased for the noise-elevated risk even when the data is not Gaussian, but has 
isotropic error covariance.   
\end{remark}

\begin{remark}
\label{rem:cb_options}
As alluded to in the proof of the proposition, various options are available in
the construction of the CB estimator; starting from \eqref{eq:cb_pred1}, we can
replace two rightmost terms by anything that has the same mean.  One might
wonder why we therefore do not just use the exact mean itself, $n\sigma^2(\alpha
- 1/\alpha)$, to define the risk estimator; as we discuss later (see Remark
\ref{rem:cb_rvar} after Proposition \ref{prop:cb_ivar}), this not a good choice,
as it would lead to a much larger variance for the risk estimator when $\alpha$
is small.        
\end{remark}

\subsection{Smoothness of noise-elevated target} 

Now that we have shown that $\CB_\alpha(g)$ is unbiased for $\Risk_\alpha(g)$,
it is natural to ask whether $\Risk_\alpha(g)$ will generally be close to the
original target of interest $\Risk(g)$.  Our next result provides a basic answer
to this question: we show that if $g$ satisfies a certain moment condition, then
the map $\alpha \mapsto \Risk_\alpha(g)$ is continuous on an interval containing
$\alpha = 0$.  In fact, if $g$ satisfies a certain $k$th order moment condition,
then this map is $k$ times continuously differentiable around $\alpha = 0$.

\begin{proposition}
\label{prop:risk_alpha_smoothness}
For $\alpha \geq 0$, let $\Risk_\alpha(g)$ be as defined in
\eqref{eq:risk_alpha}. If, for some $\beta > 0$ and integer $k \geq 0$,   
$$
\E \big[\|g(Y_\beta)\|_2^2 \|Y_\beta - \theta\|_2^{2m} \big] < \infty,
\quad m=0,\ldots,k,
$$
where recall $Y_\alpha \sim N(\theta, (1+\alpha) \sigma^2 I_n)$, then the map
$\alpha \mapsto \Risk_\alpha(g)$ has $k$ continuous derivatives on $[0, \beta)$.
\end{proposition}

The proof is not conceptually difficult but a bit technical and deferred to
Appendix \ref{app:risk_alpha_smoothness}.  It is worth noting that Proposition
\ref{prop:risk_alpha_smoothness} shows $\Risk_\alpha(g)$ is continuous in
$\alpha$ under only a moment condition, and not a continuity condition, on $g$.
Intuitively, it is reasonable to expect that continuity of $g$ would not be
needed, as evaluating the risk of $g$ at the elevated noise level $(1+\alpha) 
\sigma^2$ is akin to mollifying $g$, i.e., convolving it with a Gaussian kernel
of bandwidth $\alpha \sigma^2$, which renders the result smooth even if $g$ was
nonsmooth to begin with.

\section{Noiseless limit}
\label{sec:noiseless_limit}

Here we study the \emph{infinite-bootstrap} version of the CB estimator, 
\smash{$\CB_\alpha^\infty(g) = \lim_{B \to \infty} \CB_\alpha(g)$}.
Equivalently (by the law of large numbers), we can define this via an 
expectation over $\omega$, \smash{$\CB_\alpha^\infty(g) = \E[\CB_\alpha(g) \,|\,
  Y]$}, i.e.,
\begin{equation}
\label{eq:cb_risk_inf}
\CB_\alpha^\infty(g) = \E\big[ \|Y^\dagger - g(Y^*)\|_2^2 - \|\omega\|_2^2 /   
\alpha \,\big|\, Y \big] - n\sigma^2.
\end{equation}
where $\omega,Y^*, Y^\dagger$ denote a triplet sampled as in \eqref{eq:cb}.   
Adding and subtract $Y$ in the first quadratic term, and expanding, we get 
\begin{equation}
\label{eq:cb_risk_inf_expanded}
\CB_\alpha^\infty(g) = \E\big[ \|Y - g(Y + \sqrt\alpha \omega)\|_2^2 \,\big|\, Y
\big] + \frac{2}{\sqrt\alpha} \E\big[ \langle \omega, g(Y + \sqrt\alpha
\omega) \rangle \,\big|\, Y \big] - n\sigma^2,  
\end{equation}
where we used the fact that the inner product of $\omega$ and $Y$ has zero
conditional expectation.  

Our particular interest in this section is the behavior of
\smash{$\CB_\alpha^\infty(g)$} as $\alpha \to 0$, which we call the
\emph{noiseless limit} (referring here to the amount of auxiliary noise).  The
key is the middle term in \eqref{eq:cb_risk_inf_expanded}.  Under a moment
condition on $g$, the first term will converge the observed training error $\|Y
- g(Y)\|_2^2$, by an argument similar to that used for Proposition
\ref{prop:risk_alpha_smoothness}.  As for the middle term in
\eqref{eq:cb_risk_inf_expanded}, \citet{ramani2008monte} show that if $g$ 
admits a well-defined second-order Taylor expansion, then this same term
converges to a (scaled) divergence evaluated at $Y$: $2\sigma^2 (\nabla \cdot 
g)(Y)$.  Note that, in this case, the limit of \smash{$\CB^\infty_\alpha(g)$} as
$\alpha \to 0$ is precisely SURE in \eqref{eq:sure}. 

In fact, as Ramani et al.\ also note, the middle term in
\eqref{eq:cb_risk_inf_expanded} converges to $2\sigma^2 (\nabla \cdot g)(Y)$
even if $g$ is only weakly differentiable.  (They do not consider this extended 
case in their main paper, and refer to an online supplement for details.)  For 
completeness, we give a self-contained proof of our next result in Appendix
\ref{app:cb_noiseless}. 

\begin{theorem}
\label{thm:cb_noiseless} 
Assume the conditions of Theorem \ref{thm:sure} (Stein's result), but with the
moment conditions holding at an elevated noise level: $\E\|g(Y_\beta)\|_2^2 <
\infty$ and $\E|\nabla_i g_i(Y_\beta)| < \infty$, for $i=1,\ldots,n$, and some
$\beta > 0$.  Then the infinite-bootstrap version \eqref{eq:cb_risk_inf} of the
CB estimator (equivalently, the formulation in \eqref{eq:cb_risk_inf_expanded})
satisfies  
\begin{equation}
\label{eq:cb_noiseless}
\lim_{\alpha \to 0} \CB_\alpha^\infty(g) = \|Y - g(Y)\|^2_2 + 2 \sigma^2 (\nabla
\cdot g)(Y) = \SURE(g), \quad \text{almost surely}.
\end{equation}
Therefore, by Stein's result, the noiseless limit of
\smash{$\CB_\alpha^\infty(g)$} is unbiased for $\Risk(g)$.    
\end{theorem}

\begin{remark}
\label{rem:weak_diff}
Recall that a real-valued function $f : \R^n \to \R$ is called weakly
differentiable, with weak partial derivatives $\nabla_i f$, $i=1,\ldots,n$,
provided that for each compactly supported and continuously differentiable test  
function $\phi : \R^n \to \R$, it holds that   
\begin{equation}
\label{eq:weak_diff}
\int f(x) \nabla_i \phi(x) \, dx = -\int \nabla_i f(x) \phi(x) \, dx, \quad
i=1,\ldots,n. 
\end{equation}
Equivalently (e.g., Theorem 4.21 of \citet{evans2015measure}), a real-valued
function is weakly differentiable if it is absolutely continuous on almost every
line segment parallel to the coordinate axes. 

Meanwhile, a vector-valued function $g : \R^n \to \R^n$ is called weakly
differentiable if each of its component functions $g_i$, $i=1,\ldots,n$ are.
Equivalently, by the aforementioned ``absolute continuity on lines'' formulation
of weak differentiability, this means that for each $i=1,\ldots,n$ and
$j=1,\ldots,n$,
$$
y_i \mapsto g_j(y) \;\, \text{is absolutely continuous on compact subsets of
  $\R$, for almost every $y_{-i} \in \R^{n-1}$},
$$
where $y_{-i}$ denotes the vector $y$ with the $i$th component removed.  This is
a stronger condition than what is really required in Theorems \ref{thm:sure} or
\ref{thm:cb_noiseless}.  Each result in fact only requires that for each
$i=1,\ldots,n$,   
$$
y_i \mapsto g_i(y) \;\, \text{is absolutely continuous on compact subsets of
  $\R$, for almost every $y_{-i} \in \R^{n-1}$}.
$$
Effectively, each component function $g_i$ only needs to be weakly
differentiable with respect to the $i$th variable (not all of the other
variables), for almost every choice of $y_{-i} \in \R^{n-1}$.  While this is
technically weaker than weak differentiability, it is also harder to explain,
and not clear whether this distinction is all that meaningful.  For simplicity,
we thus state the assumption as weak differentiability of $g$ in both Theorems
\ref{thm:sure} and \ref{thm:cb_noiseless}. 
\end{remark}

\begin{remark}
The limiting result in \eqref{eq:cb_noiseless} also holds for the
infinite-bootstrap version of the BY estimator in \eqref{eq:by_risk}, which is
the estimator studied in \citet{ramani2008monte} (as we mentioned in Section
\ref{sec:related_work}, these authors seemed to be unaware of the prior work of
Breiman and Ye, and independently proposed the same estimator). The
infinite-bootstrap formulation of the BY estimator, \smash{$\BY_\alpha^\infty(g) 
  = \E[\BY_\alpha(g) \,|\, Y]$}, can be expressed as
\begin{equation}
\label{eq:by_risk_inf_expanded}
\BY_\alpha^\infty(g) = \|Y - g(Y)\|_2^2 + \frac{2}{\sqrt\alpha} \E\big[
\langle \omega, g(Y + \sqrt\alpha \omega) \rangle \,\big|\, Y \big] - n\sigma^2, 
\end{equation}
which is very similar to the analogous representation
\eqref{eq:cb_risk_inf_expanded} for the infinite-bootstrap CB estimator. From
this $B=\infty$ perspective, the only difference between the estimators 
\eqref{eq:cb_risk_inf_expanded} and \eqref{eq:by_risk_inf_expanded} is the first
term; and as $\alpha \to 0$, the first term in \eqref{eq:cb_risk_inf_expanded}
will converge to that in \eqref{eq:by_risk_inf_expanded} if $\E\|g(Y)\|_2^2 <
\infty$ (even for nonsmooth $g$).
\end{remark}

\begin{remark}
The fact that SURE can be recreated as a certain limiting case of the CB
estimator is of course reassuring, as the former is one of the most celebrated
and well-studied ideas in risk estimation in the normal means model. When $g$ is
weakly differentiable but its divergence is not analytically tractable, using
the CB estimator for small but nonzero $\alpha > 0$ (which can be seen as
employing a particular numerical approximation scheme for the divergence) is
appealing, because we know that it should behave similarly to SURE for small
enough $\alpha > 0$; and yet for any $\alpha > 0$, we know that it targets the
noise-elevated risk $\Risk_\alpha(g)$.
\end{remark}

\begin{remark}
\label{rem:cb_noiseless_ht}
What happens in the limit when $g$ is not weakly differentiability? While it
still may be the case that the noiseless limit of the infinite-bootstrap CB
estimator is SURE, it will no longer be generally true that this noiseless limit
is unbiased for $\Risk(g)$, because the unbiasedness of SURE itself requires
weak differentiability of $g$. (Of course, the same can be said about the
infinite-bootstrap BY estimator, since, as the above remark explains, it has the
same limit as $\alpha \to 0$.)

As a simple example, consider the hard-thresholding estimator $g$, which is
discontinuous (and not weakly differentiable), with component functions $g_i(Y)
= Y_i \cdot 1\{|Y_i| \geq t\}$, $i=1,\ldots,n$, for some fixed $t > 0$. In this 
case, we can check by direct computation (see Appendix
\ref{app:cb_noiseless_ht}) that            
\begin{equation}
\label{eq:cb_noiseless_ht}
\lim_{\alpha \to 0} \, \frac{2}{\sqrt\alpha} \E\big[ \langle \omega, g(y +
\sqrt\alpha \omega) \rangle \big] = 2\sigma^2 \sum_{i=1}^n 1\{|y_i| \geq t\}, 
\quad \text{for almost every $y \in \R^n$}.
\end{equation}
The right-hand side is again the (scaled) divergence of $g$ evaluated at $y$,
which is well-defined for almost every $y$; however, it is known that the
divergence does \emph{not} lead to an unbiased estimate of risk for
hard-thresholding, due to the discontinuous nature of this estimator; see, e.g.,
\citet{tibshirani2015degrees}.       
\end{remark}

\section{Bias and variance}
\label{sec:bias_variance}

In this section, we analyze a bias-variance decomposition of the mean squared 
error of $\CB_\alpha(g)$ in \eqref{eq:cb_risk}, when we measure its error to the
\emph{original} risk $\Risk(g)$.  For any estimator \smash{$\hat{R}(g)$} of
$\Risk(g)$, recall:  
$$
\E[\hat{R}(g) - \Risk(g)]^2 = 
\underbrace{\big[\E[\hat{R}(g)] - \Risk(g)\big]^2}_{\Bias^2(\hat{R}(g))} \,+\,   
\underbrace{\E\big[\hat{R}(g) - \E[\hat{R}(g)]\big]^2}_{\Var(\hat{R}(g))}. 
$$
Applying this decomposition to the CB estimator $\CB_\alpha(g)$, we get:   
\begin{equation}
\label{eq:cb_bias_var}
\E[\CB_\alpha(g) - \Risk(g)]^2 = 
\underbrace{\big[\Risk_\alpha(g) - \Risk(g)\big]^2}_{\Bias^2(\CB_\alpha(g))}
\,+\,
 \underbrace{\E\big[\Var(\CB_\alpha(g) \,|\, Y)\big]}_{\RVar(\CB_\alpha(g))}
 \,+\, 
\underbrace{\Var\big(\E[\CB_\alpha(g) \,|\, Y]\big)}_{\IVar(\CB_\alpha(g))}.
\end{equation}
Here, for the bias term, we used the fact that $\CB_\alpha(g)$ is unbiased for
the noise-elevated risk $\Risk_\alpha(g)$ from Corollary \ref{cor:cb_unbiased};
and for the variance term, we used the law of total variance, and denote the two
terms that fall out by $\RVar(\CB_\alpha(g))$ (expectation of the conditional 
variance) and $\IVar(\CB_\alpha(g))$ (variance of the conditional expectation),
which we will call the \emph{reducible} and \emph{irreducible} variance of
$\CB_\alpha(g)$, respectively.  This is meant to reflect the effect of the
number of bootstrap draws $B$: the reducible variance will shrink as $B$ grows,
but the irreducible variance does not depend depend on $B$ at all, and in fact,
it can be viewed as the variance of the infinite-bootstrap version of the risk
estimator, \smash{$\CB_\alpha^\infty(g) = \E[\CB_\alpha(g) \,|\, Y]$}.

The goal of this section is to develop a precise understanding of how the
individual terms in \eqref{eq:cb_bias_var} behave as a function of $\alpha$ and
$B$, with a particular focus on small $\alpha$ and large $B$. As usual, we assume
throughout that $Y \sim N(\theta, \sigma^2 I_n)$, and recall we denote $Y_\alpha
\sim N(\theta, (1+\alpha) \sigma^2 I_n)$ for $\alpha \geq 0$.    

\subsection{Bias}

The next result provides an exact expression for $\Bias(\CB_\alpha(g)) = 
\Risk_\alpha(g) - \Risk(g)$, and some bounds for its magnitude.   

\begin{proposition}
\label{prop:cb_bias}
Assume $\E[\|g(Y_\beta)\|_2^2 \|Y_\beta - \theta\|_2^{2m}] < \infty$ for 
$m=0,1$ and some $\beta > 0$.  Then for all $\alpha \in [0, \beta)$, 
\begin{equation}
\label{eq:cb_bias}
\Risk_\alpha(g) - \Risk(g) = \int_0^\alpha \frac{\sqrt{n}}{\sqrt{2}(1+t)}
\sqrt{\Var(\|\theta - g(Y_t)\|_2^2)} \, \Cor\big(\|\theta - g(Y_t)\|_2^2,    
\|Y_t - \theta\|_2^2\big) \, dt. 
\end{equation}
If $\Var(\|\theta - g(Y_t)\|_2^2)$ is increasing with $t$ on $[0, \alpha]$, then
a simple upper bound is  
\begin{equation}
\label{eq:cb_bias_bd1}
|\Risk_\alpha(g) - \Risk(g)| \leq \frac{\sqrt{n}\alpha}{\sqrt{2}}
\sqrt{\Var(\|\theta - g(Y_\alpha)\|_2^2)}.
\end{equation}
If in addition $\E[\|g(Y_\beta)\|_2^4 \|Y_\beta - \theta\|_2^{2m}] < \infty$
for $m=0,1$, then for all $\alpha \in [0, \beta)$, 
\begin{equation}
\label{eq:cb_bias_bd2}
|\Risk_\alpha(g) - \Risk(g)| \leq \frac{\sqrt{n}\alpha}{\sqrt{2}}
\sqrt{\Var(\|\theta - g(Y)\|_2^2)} + O(\alpha^{3/2}),
\end{equation}
where here and throughout, we use the asymptotic notation $f(\alpha) =
O(h(\alpha))$ to mean that there is a constant $C>0$ such that $f(\alpha) \leq C
h(\alpha)$ for small enough $\alpha$.  
\end{proposition}

The proof of Proposition \ref{prop:cb_bias} is deferred to Appendix
\ref{app:bias_var}.  The upper bound in \eqref{eq:cb_bias_bd1} shows that the
absolute bias has a near-linear decay with $\alpha$, where ``near'' reflects
that  \smash{$\Var(\|\theta - g(Y_\alpha)\|_2^2)$} also depends on $\alpha$.
Under additional moment conditions on $g$, we see from \eqref{eq:cb_bias_bd2}
that the bias indeed decays linearly with $\alpha$.  Empirical examples that
assess the bias bounds from Proposition \ref{prop:cb_bias} are given in Appendix 
\ref{app:additional_experiments}.

\begin{remark}
\label{rem:cb_bias_rel}
With regard to the bound in \eqref{eq:cb_bias_bd2}, observe that   
\begin{equation}
\label{eq:var_jensen}
\sqrt{\Var(\|\theta - g(Y)\|_2^2)} \leq \sqrt{\E\|\theta - g(Y)\|_2^4} \leq
\Risk(g), 
\end{equation}
with the last step holding by Jensen's inequality, and thus to leading order,
we can interpret \eqref{eq:cb_bias_bd2} as providing for us an upper bound on
the \emph{relative} bias: 
\begin{equation}
\label{eq:cb_bias_bd3}
\frac{|\Risk_\alpha(g) - \Risk(g)|}{\Risk(g)} \lesssim \frac{\sqrt{n}
  \alpha}{\sqrt{2}}, 
\end{equation}
where $\lesssim$ means that we omit all terms with a lower-order dependence on 
$\alpha$.  This suggests that to achieve a relative bias of $x$, we
should choose \smash{$\alpha = \sqrt{2}x/\sqrt{n}$} (e.g., for $x = 10\%$, we
set \smash{$\alpha \approx 14/\sqrt{n}$}). 

We remark that \eqref{eq:cb_bias_bd3} will be often conservative in
practice. This is due to of looseness in the inequality
\smash{$\sqrt{\Var(\|\theta - g(Y)\|_2^2)} \leq \Risk(g)$} derived in
\eqref{eq:var_jensen}, and looseness in the bound $\Cor(\|\theta - g(Y_t)\|_2^2,
\|Y_t - \theta\|_2^2) \leq 1$ used to derive \eqref{eq:cb_bias_bd1},
\eqref{eq:cb_bias_bd2}. For example, when $g(Y) = SY$ and $S$ projects onto a 
$p$-dimensional linear suspace (as in linear regression), one can check that 
$$
\sqrt{\Var(\|\theta - g(Y)\|_2^2)} = \sqrt{2p}\sigma^2 \ll 
\|\theta - S\theta\|_2^2 + p\sigma^2 = \Risk(g) \quad 
\text{when $p \ll n$ (or $\theta$ is far from $S\theta$)},
$$
and
$$
\Cor\big(\|\theta - g(Y_t)\|_2^2, \|Y_t - \theta\|_2^2\big) = \sqrt{p/n} \ll 1  
\quad \text{when $p \ll n$}.
$$
\end{remark}

\subsection{Reducible variance}

The next result gives a simple bound on the reducible variance
$\RVar(\CB_\alpha(g))$. 

\begin{proposition}
\label{prop:cb_rvar}
Assume $\E\|g(Y_\beta)\|_2^4 < \infty$ for some $\beta > 0$.  Then for all
$\alpha \in (0, \beta)$, 
\begin{equation}
\label{eq:cb_rvar_bd}
\RVar(\CB_\alpha(g)) = \frac{4\sigma^2}{B\alpha} \E\|Y - g(Y)\|_2^2 +
O\bigg(\frac{1}{B\sqrt\alpha}\bigg).   
\end{equation}
\end{proposition}

The proof of Proposition \ref{prop:cb_rvar} is in Appendix 
\ref{app:bias_var}. Empirical examples that investigate the reducible variance
and the dominance of the leading term \smash{$4\sigma^2\E\|Y - g(Y)\|_2^2 /
  (B\alpha)$} in \eqref{eq:cb_rvar_bd} are given in Appendix
\ref{app:additional_experiments}.    

\begin{remark}
\label{rem:cb_rvar}
The dependence of the leading term in \eqref{eq:cb_rvar_bd} on $\alpha$, which
scales as $1/\alpha$, is a consequence of a careful construction in the CB
estimator. Recall that in Remark \ref{rem:cb_options}, we explained that various
options can be used in place of the last two terms in \eqref{eq:cb_pred1}.  One
can check that choosing the exact mean $n\sigma^2(\alpha - 1/\alpha)$ would lead
to an estimator with irreducible variance $2n\sigma^4/(B\alpha^2) +
O(1/(B\alpha))$, whose leading term scales as $1/\alpha^2$.  This is due to the
conditional variance of $\|Y^\dagger\|_2^2$ given $Y$, where \smash{$Y^\dagger =
  Y - \omega/\sqrt\alpha$} is as in \eqref{eq:cb}.  Both of the options in
\eqref{eq:cb_pred1} and \eqref{eq:cb_pred2}---the second one here being the
basis for the CB estimator---substantially improve upon this, bringing the order
of dependence down to $1/\alpha$, as they each subtract off a term that 
effectively cancels out the variation of $\|Y^\dagger\|_2^2$.  The differences
between \eqref{eq:cb_pred1} and \eqref{eq:cb_pred2} are much less pronounced;
the former yields a reducible variance with leading term
\smash{$4\sigma^2\E\|g(Y)\|_2^2 / (B\alpha)$}, whereas the latter yields a
reducible variance \eqref{eq:cb_rvar_bd} with leading term
\smash{$4\sigma^2\E\|Y - g(Y)\|_2^2 / (B\alpha)$}, which can often be
smaller. For this reason, we choose to define the CB estimator as we did, based
on \eqref{eq:cb_pred2}. 
\end{remark}

\begin{remark}
For the BY estimator in \eqref{eq:by_risk}, the same arguments as in the proof
of Proposition \ref{prop:cb_rvar} show that, under the same conditions on $g$,
the reducible variance satisfies
\begin{equation}
\label{eq:by_rvar_bd}
\RVar(\BY_\alpha(g)) = \frac{4\sigma^2}{B\alpha} \E\|g(Y)\|_2^2 +
O\bigg(\frac{1}{B\sqrt\alpha}\bigg).    
\end{equation}
Note that the order of dependence here is $1/\alpha$, as in the CB
estimator. However, the factor $\E\|g(Y)\|_2^2$ that multiplies the leading
order in \eqref{eq:by_rvar_bd} can often be larger than the factor $\E\|Y - 
g(Y)\|_2^2$ in \eqref{eq:cb_rvar_bd} (as just noted at the end of the last
remark). 
\end{remark}

\begin{remark}
If we are using risk estimation to choose between models (functions) $g$ and
\smash{$\tilde{g}$}, where each of these satisfy the conditions of Proposition
\ref{prop:cb_rvar}, and importantly, we use the same bootstrap draws in
\eqref{eq:cb} for constructing $\CB_\alpha(g)$ and
\smash{$\CB_\alpha(\tilde{g})$}, then the same arguments as in the proof of
Proposition \ref{prop:cb_rvar} show that  
\begin{equation}
\label{eq:cb_rvar_diff_bd}
\RVar\big(\CB_\alpha(g) -\CB_\alpha(\tilde{g})\big) = \frac{4\sigma^2}{B\alpha}
\E\|g(Y) - \tilde{g}(Y)\|_2^2 + O\bigg(\frac{1}{B\sqrt\alpha}\bigg). 
\end{equation}
Note that the factor in \smash{$\E\|g(Y) - \tilde{g}(Y)\|_2^2$} multiplying  
the leading order in \eqref{eq:cb_rvar_diff_bd} can be even smaller than
the factor $\E\|Y - g(Y)\|_2^2$ in \eqref{eq:cb_rvar_bd}, when $g$ and 
\smash{$\tilde{g}$} are similar. 
\end{remark}  

\subsection{Irreducible variance}
\label{sec:ivar}

Recalling the expression for the infinite-bootstrap version of the CB estimator
in \eqref{eq:cb_risk_inf_expanded}, observe that we can always write the
irreducible variance, for any $g$ and any $\alpha > 0$, as
\begin{equation}
\label{eq:cb_ivar}
\IVar(\CB_\alpha(g)) = \Var\bigg( \E\big[ \|Y - g(Y + \sqrt\alpha \omega)\|_2^2 
\,\big|\, Y \big] + \frac{2}{\sqrt\alpha} \E\big[ \langle \omega, g(Y +
\sqrt\alpha \omega) \rangle \,\big|\, Y \big] \bigg).
\end{equation}
The following result studies the behavior of $\IVar(\CB_\alpha(g))$ for small
$\alpha$, under a suitable condition on $g$.  

\begin{proposition}
\label{prop:cb_ivar}
Assume that 
\begin{equation}
\label{eq:h_limit}
h(y) = \lim_{\alpha \to 0} 
\frac{2}{\sqrt\alpha} \E[\langle \omega, g(y + \sqrt\alpha \omega) \rangle]  
\quad \text{exists for almost every $y \in \R^n$},   
\end{equation}
and this convergence comes with a dominating function $H$ with $\E[H(Y)] < 
\infty$ such that   
\begin{equation}
\label{eq:h_dominate}
\frac{4}{\alpha} \E\big[\langle \omega, g(y + \sqrt\alpha \omega) \rangle
\big]^2 \leq H(y) \quad \text{for almost every $y \in \R^n$ and $\alpha \leq  
\beta$}, 
\end{equation}
for some $\beta > 0$.  Assume also that $g$ satisfies $\E\|g(Y_\beta)\|_2^4 <
\infty$.  Then     
\begin{equation}
\label{eq:cb_ivar_bd}
\IVar(\CB_\alpha(g)) = \Var\big(\|Y - g(Y)\|_2^2 + h(Y)\big) + o(1),
\end{equation}
where $o(1)$ denotes a term that converges to zero as $\alpha \to 0$. 
\end{proposition}

The proof of Proposition \ref{prop:cb_ivar} is in Appendix \ref{app:bias_var}.
Empirical examples that examine the irreducible variance for small $\alpha$ can
be found in Appendix \ref{app:additional_experiments}.  

\begin{remark}
For the BY estimator, recall, its infinite-bootstrap version takes the form
\eqref{eq:by_risk_inf_expanded}, which means that its irreducible variance is  
\begin{equation}
\label{eq:by_ivar}
\IVar(\BY_\alpha(g)) = \Var\bigg( \E[ \|Y - g(Y)\|_2^2] + \frac{2}{\sqrt\alpha}
\E\big[ \langle \omega, g(Y + \sqrt\alpha \omega) \rangle \,\big|\, Y \big]
\bigg). 
\end{equation}
This is just as in \eqref{eq:cb_ivar}, but in the first term (inside of the
variance), we are measuring the error between $Y$ and $g(Y)$, rather than
$Y$ and $g$ applied to the noise-elevated data.  The result of
Proposition \ref{prop:cb_ivar} carries over to the BY estimator: under
\eqref{eq:h_limit}, \eqref{eq:h_dominate}, and the moment condition on $g$,
the same small-$\alpha$ representation in \eqref{eq:cb_ivar_bd} holds for
$\IVar(\BY_\alpha(g))$.  The subtle difference between \eqref{eq:cb_ivar} and
\eqref{eq:by_ivar} can indeed materialize in practice, especially when the
estimate $g$ is nonsmooth and unstable. See Figure \ref{fig:risk_comparison2}
and the discussion in Section \ref{sec:comp_cb_by}. 
\end{remark} 

\begin{remark}
\label{rem:cb_ivar_wd}
As we showed in Theorem \ref{thm:cb_noiseless} (a similar result appears in
\citet{ramani2008monte}), when $g$ is weakly differentiable and its weak partial
derivatives are integrable, the limit in \eqref{eq:h_limit} exists, and equals     
$$
h(y) = \sigma^2 (\nabla \cdot g)(y) = 2 \sigma^2 \sum_{i=1}^n \nabla_i g_i(y),  
$$
which is the divergence of $g$ (scaled by $2\sigma^2$).  Furthermore, one can 
check that condition \eqref{eq:h_dominate} is implied by squared integrability
of the divergence at an elevated noise level: $\E[(\nabla \cdot g)(Y_\alpha)^2]
< \infty$ for some $\alpha > 0$. The result in \eqref{eq:cb_ivar_bd} then reads   
$$
\IVar(\CB_\alpha(g)) = \Var\big(\|Y - g(Y)\|_2^2 + 2\sigma^2 (\nabla \cdot
g)(Y)\big) + o(1),
$$
i.e., the irreducible variance of the CB estimator converges to the variance of
SURE, as $\alpha \to 0$. 
\end{remark}

\begin{remark}
\label{rem:cb_ivar_ht}
It is worth emphasizing that the dominating condition in \eqref{eq:h_dominate}
is key: without it, the result in the proposition is not true in general.  As an
example, consider the hard-thresholding function, which, recall, has components
$g_i(Y) = Y_i \cdot 1\{|Y_i| \geq t\}$, $i=1,\ldots,n$. This satisfies the limit
condition in \eqref{eq:h_limit}, where the limiting function $h$ is $2\sigma^2
\nabla \cdot g$, as in \eqref{rem:cb_noiseless_ht}.  However, in a sense we
already know that the limiting irreducible variance of hard-thresholding should
not simply be the variance of SURE, due to the bias of SURE for the risk in this
case \citep{tibshirani2015degrees}. Indeed, a direct calculation (building off
that in Appendix \ref{app:cb_noiseless_ht}) confirms that \eqref{eq:h_dominate}
fails for the hard-thresholding function. 

The importance of the dominating condition \eqref{eq:h_dominate} raises a
natural question: what are the most general conditions under which
\eqref{eq:h_dominate} holds? Can it be met outside of weak differentiability of
$g$? As of now, we do not have a good answer to this, and it remains an open
question for future work. 
\end{remark}

\subsection{Summary of bias and variance results}

\begin{table}[tb]
\centering\small
\def\arraystretch{1.5} 
\begin{tabular}{c|c|c|c}
& Bias to $\Risk(g)$ & Reducible variance & Irreducible variance \\
\hline 
CB estimator &  
$\lesssim \alpha \sqrt{\frac{n}{2} \Var(\|\theta - g(Y)\|_2^2)}$ &
$\lesssim \frac{4\sigma^2}{B\alpha} \E\|Y - g(Y)\|_2^2$ &
stable when $g$ is smooth \\
BY estimator & ? & 
$\lesssim \frac{4\sigma^2}{B\alpha} \E\| g(Y)\|_2^2$ &
stable when $g$ is smooth \\
\end{tabular}
\caption{\small\it Summary of bias and variance results described across
  Propositions \ref{prop:cb_bias}--\ref{prop:cb_ivar} and ensuing remarks.   
  Above, $\lesssim$ means that we omit all terms with a lower-order dependence
  on $\alpha$.}
\label{tab:bv_summary}
\end{table}

Table \ref{tab:bv_summary} summarizes the bias and variance results from this
section. We use ``stable when $g$ is smooth'' for the irreducible variance to
reflect the fact that it is not clear in what general settings this will be
stable as $\alpha \to 0$, for the CB and BY methods; recall, for either method,
the conditions in \eqref{eq:h_limit}, \eqref{eq:h_dominate} are sufficient to
ensure that the limiting irreducible variance satisfies \eqref{eq:cb_ivar}.
While these conditions are met (and the limiting irreducible variance is the
variance of SURE) in the case of weakly differentiable $g$ (Remark
\ref{rem:cb_ivar_wd}), the extent to which these conditions apply beyond weak
differentiability remains unclear, and for hard-thresholding as a key non-weakly
differentiable example, the second condition fails (Remark
\ref{rem:cb_ivar_ht}).

The lack of clarity on the irreducible variance prevents us from reasoning
holistically about the behavior of the CB or BY methods in the infinitesimal
$\alpha$ regime (beyond the case of smooth $g$).  However, practically speaking,
for a given data set at hand, we would of course choose $\alpha$ to be small but
non-infinitesimal, such as $\alpha=0.01$, or $\alpha=0.05$.  This brings us to a
primary advantage of the CB estimator in particular, reflected in the first
column of the table: it is always unbiased for $\Risk_\alpha(g)$, the risk of
$g$ at the noise-elevated level of $(1+\alpha) \sigma^2$.  Therefore, provided
that we have a sense---practically, conceptually, or theoretically (first
column, see also Remark \ref{rem:cb_bias_rel})---that $\Risk_\alpha(g)$ is a
reasonable target of estimation, we do not have to concern ourselves with the
infinitesimal $\alpha$ regime.

\section{Experiments}
\label{sec:experiments}

In this section, we study the performance of the CB method empirically.  The
first two subsections compare the CB and BY estimators in simulations (the
performance for Efron's method was generally much worse and its results were
omitted accordingly). The third studies the use of the CB estimator for
parameter tuning in an image denoising application. Code to reproduce all
experimental results in this section is available online at
\url{https://github.com/nloliveira/coupled-bootstrap-risk-estimation}.    

\subsection{Comparison of CB and BY}
\label{sec:comp_cb_by}

We compare the CB estimator \eqref{eq:cb_risk} to the BY estimator
\eqref{eq:by_risk} in simulations. At a high level, our goal is not to show that
CB is better than BY at estimating risk in every scenario, but instead to
provide experimental support around the following four points: 
\begin{enumerate}[(i)]
\item for any $g$ and any $\alpha$, $\CB_\alpha(g)$ is unbiased for
  $\Risk_\alpha(g)$;  
\item for linear $g$ and any $\alpha$, $\BY_\alpha(g)$ is unbiased for
  $\Risk(g)$;
\item for nonlinear $g$, the bias of $\BY_\alpha(g)$ is unpredictable---it can
  be increasing or decreasing as $\alpha$ increases, and it can be larger or 
  smaller than $\Risk_\alpha(g)$; 
\item for unstable $g$, the variance of $\BY_\alpha(g)$ can be much larger than
  that of $\CB_\alpha(g)$. 
\end{enumerate}

Throughout, we fix $n=100$ and $p=200$, and we generate data $Y \in \R^n$ from
a linear model with feature matrix $X \in \R^{n \times p}$. At the outset, we
draw the entries of $X$ from $N(0,1)$, and we draw the coefficient vector in the
linear model $\beta \in \R^p$ to have $s$ nonzero entries from
$\mathrm{Unif}(-1,1)$.  The features $X$ and coefficient vector $\beta$ are then
fixed for all subsequent repetitions of the given simulation.  For each
repetition $r=1,\ldots,100$, we generate a response vector    
$$
Y^{(r)} = X \beta + \epsilon^{(r)}, 
$$
where the error vector \smash{$\epsilon^{(r)} \in \R^n$} has i.i.d.\ entries
from $N(0, \sigma^2$), and the error variance $\sigma^2$ is chosen to meet a
desired signal-to-noise ratio \smash{$\mathrm{SNR} = \Var_n(X\beta) / \sigma^2$} 
(where $\Var_n(\cdot)$ denotes the empirical variance operator on $n$
samples). We then apply each risk estimator (CB or BY) to \smash{$Y^{(r)}$},
with a particular function $g$, number of bootstrap draws $B$, and auxiliary
noise parameter $\alpha$, in order to produce a risk estimate.  Finally, we
report aggregate results over all repetitions $r=1,\ldots,100$. 

The number of bootstrap draws is fixed at $B=100$ throughout.  We consider
four different functions $g$: 
(a) ridge regression, with a fixed tuning parameter, $\lambda=5$;
(b) lasso regression, with a fixed tuning parameter, $\lambda=0.31$;
(c) forward stepwise regression, with a fixed number of steps, $k=2$; and
(d) lasso regression, with the tuning parameter $\lambda$ chosen by
cross-validation. (The particular tuning parameter values $\lambda$ for ridge
and lasso were chosen because they were close to the middle, roughly speaking,
of their effective solution paths.)  Our implementation uses the \texttt{glmnet}
\citep{friedman2010regularization} R package for ridge and lasso, and
\texttt{bestsubset} \citep{hastie2020best} R package for forward stepwise. We
note that the functions $g$ in (a) and (b) are weakly differentiable, but those
in (c) and (d) are not.  For CB, we consider six values for $\alpha$: 0.05, 0.1,
0.2, 0.5, 0.8, and 1.  All risks and risk estimates throughout are scaled by
$1/n$.

\begin{figure}[htbp]
\centering
\begin{subfigure}[b]{0.425\textwidth}
\includegraphics[width=\textwidth]{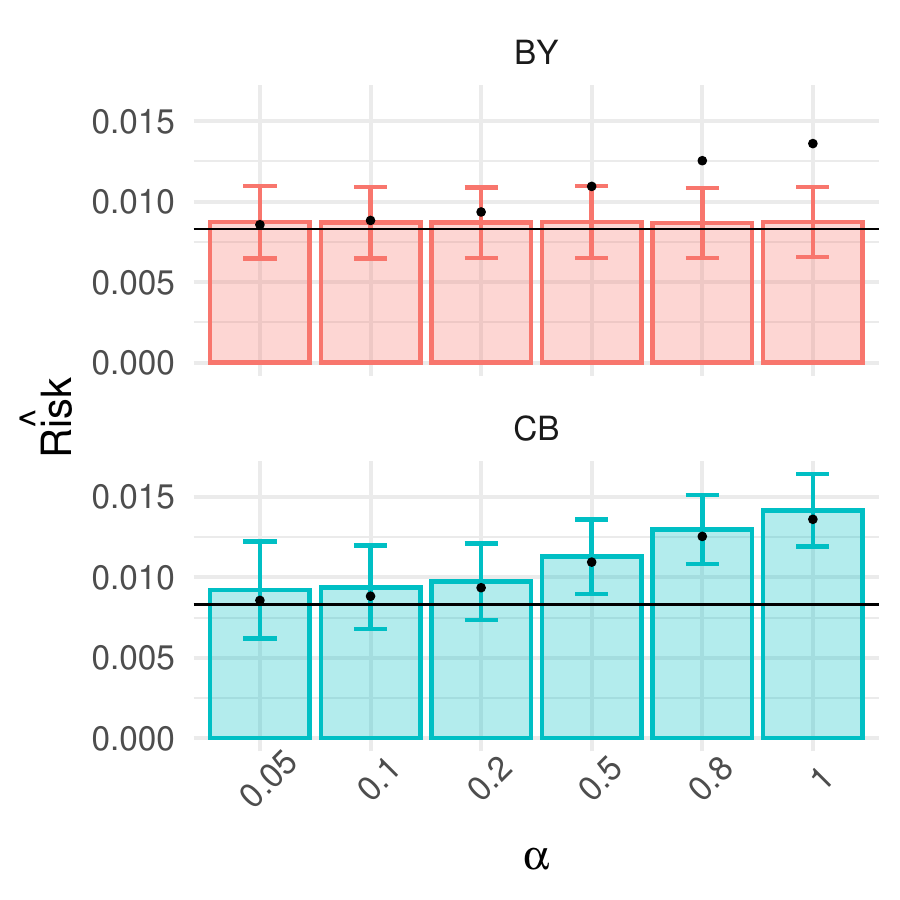}
\caption{Ridge regression, fixed $\lambda$}
\end{subfigure}
\begin{subfigure}[b]{0.425\textwidth}
\includegraphics[width=\textwidth]{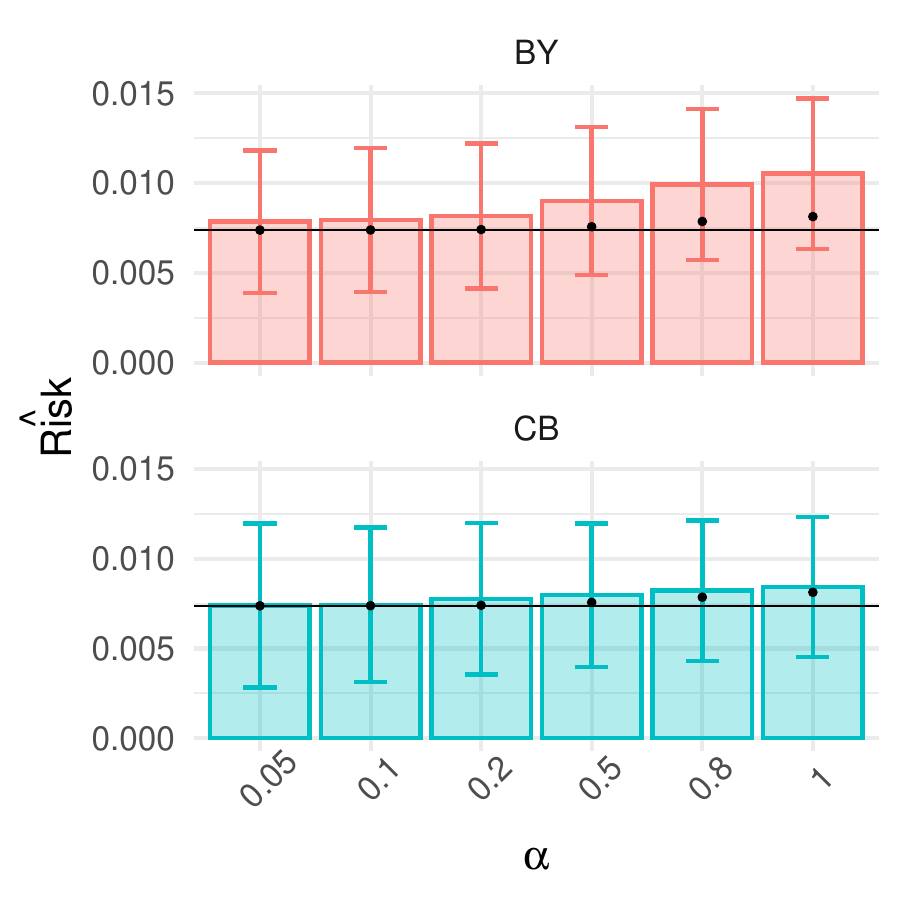}
\caption{Lasso regression, fixed $\lambda$}
\end{subfigure} \\
\begin{subfigure}[b]{0.425\textwidth}
\includegraphics[width=\textwidth]{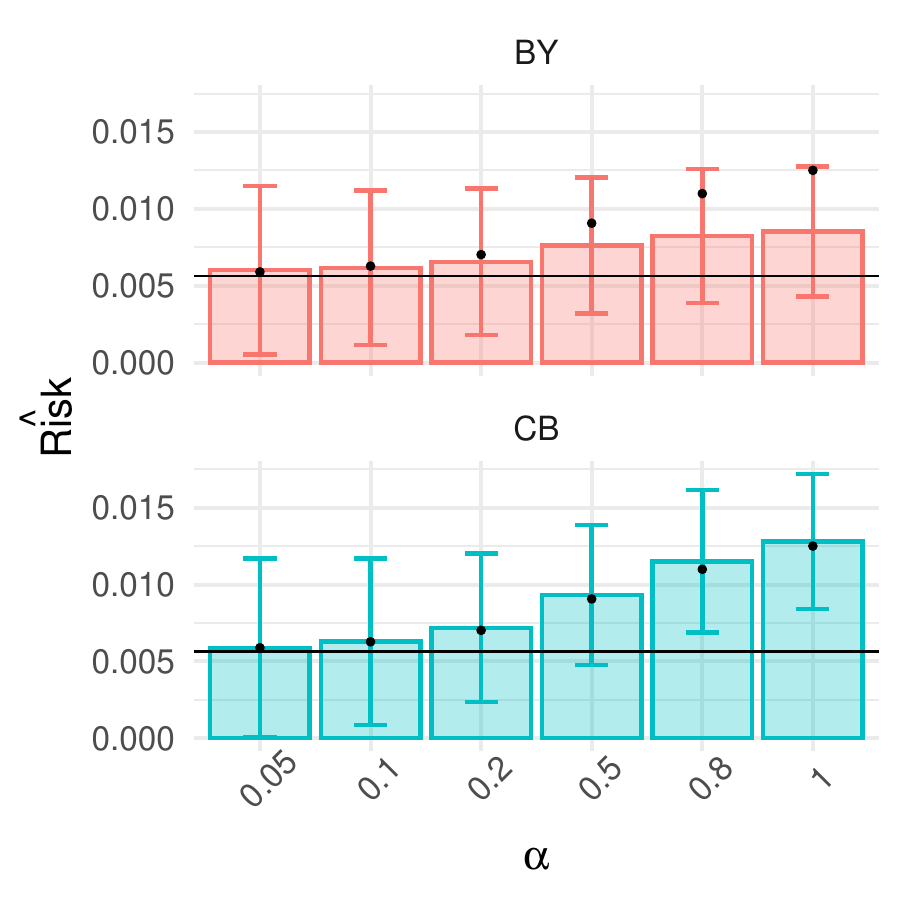}
\caption{Forward stepwise, fixed $k$}
\end{subfigure}
\begin{subfigure}[b]{0.425\textwidth}
\includegraphics[width=\textwidth]{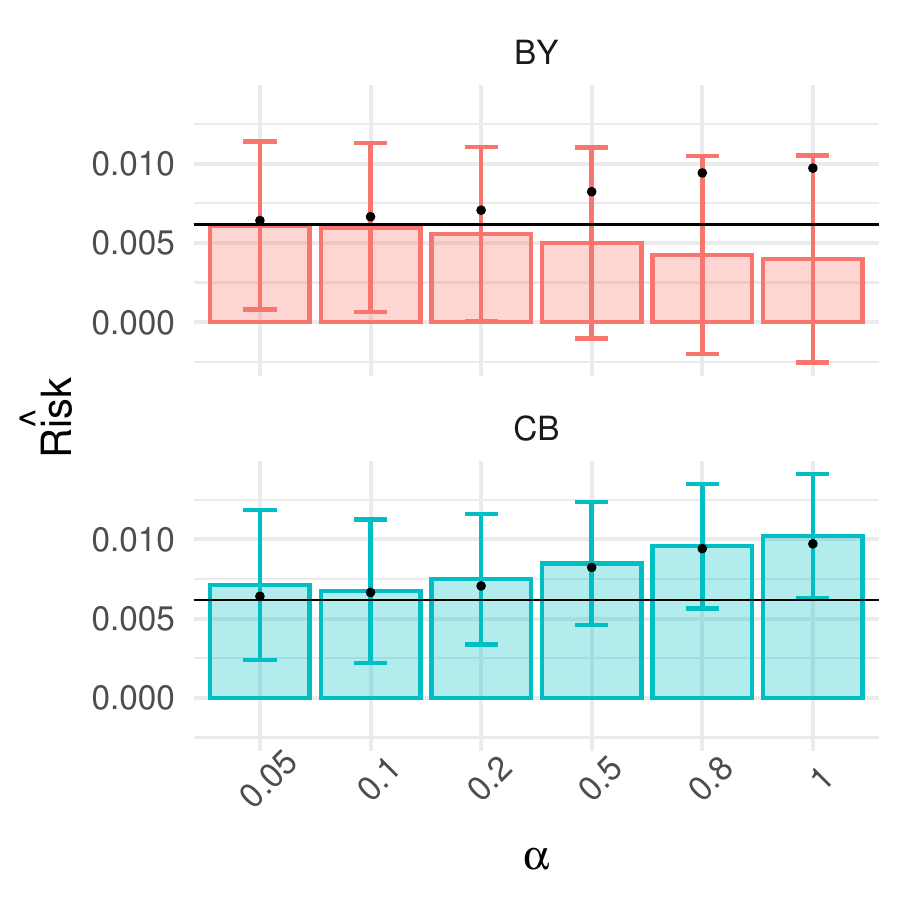}
\caption{Lasso regression, CV-tuned $\lambda$}
\end{subfigure}
\caption{Comparison of CB and BY risk estimators for different functions $g$,
  when $s=5$ and $\mathrm{SNR}=0.4$.}
\label{fig:risk_comparison1}

\bigskip
\centering
\includegraphics[width=0.7\textwidth]{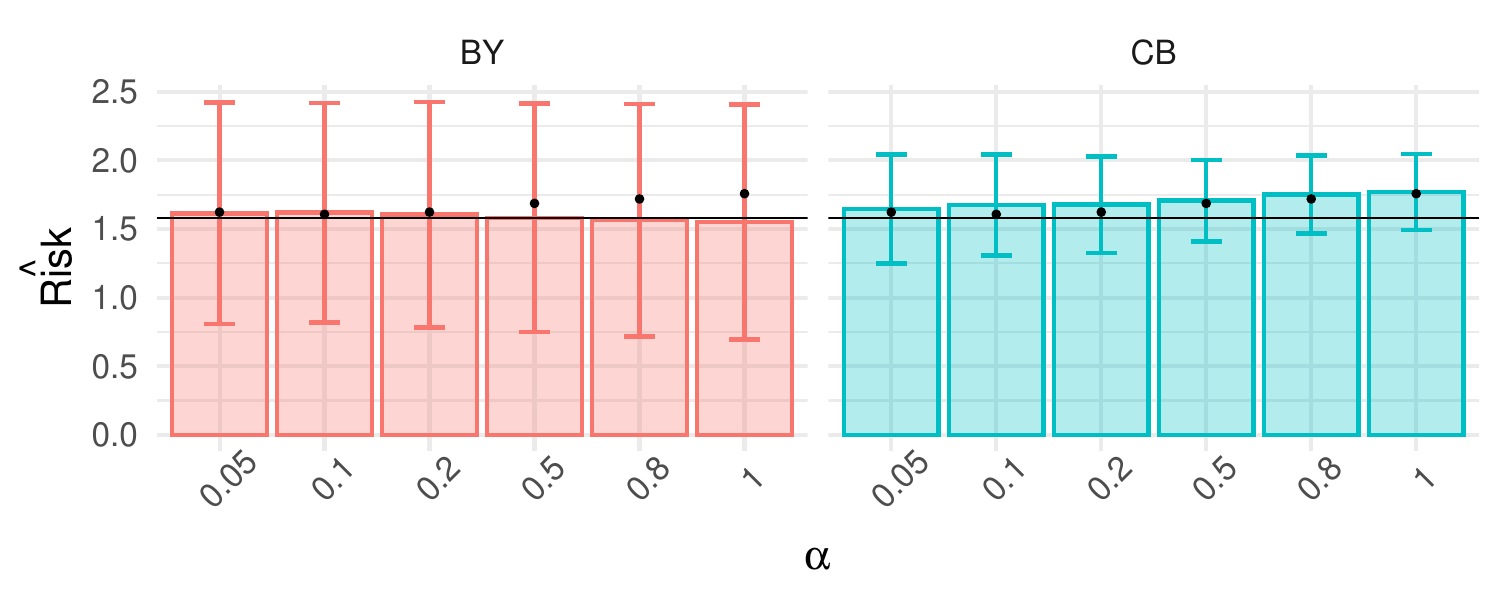}
\caption{Comparison of CB and BY risk estimators when $g$ is the lasso tuned by 
  cross-validation, $s=200$, and $\mathrm{SNR}=2$.} 
\label{fig:risk_comparison2}
\end{figure}

Figure \ref{fig:risk_comparison1} shows the results for when the underlying
linear model has sparsity $s=5$, and $\mathrm{SNR}=0.4$.  The figure displays
the average risk estimate from each method, CB and BY, as well as standard
errors of these risk estimates. Each panel (a)--(d) corresponds to one of the
four functions $g$ described above. In each panel, the black horizontal line
represents $\Risk(g)$, and the black dots represent $\Risk_\alpha(g)$ (which are
themselves estimated via Monte Carlo).  We can see that for each function $g$,
the CB method is unbiased for $\Risk_\alpha(g)$, as expected.  Meanwhile, we see
that the bias of the BY method varies quite a lot with $g$, having zero,
positive, or negative bias, depending on the situation. In panel (a), where $g$
is a linear smoother (ridge), the average BY estimate matches $\Risk(g)$, for
any $\alpha$, as expected. In (b), where $g$ is nonlinear but weakly
differentiable (lasso), it overestimates $\Risk_\alpha(g)$ and thus also
$\Risk(g)$, more so at the larger values of $\alpha$.  In panel (c), where $g$
is both nonlinear and nonsmooth (forward stepwise), BY underestimates
$\Risk_\alpha(g)$ but also overestimates $\Risk(g)$, for larger $\alpha$; and in
(d), where $g$ is again nonlinear and nonsmooth (lasso tuned by
cross-validation), it underestimates both $\Risk_\alpha(g)$ and $\Risk(g)$ for
larger $\alpha$.  In summary, there is no single consistent behavior for the
bias of $\BY_\alpha(g)$ across all scenarios.  (Certainly, the rule-of-thumb 
advocated by \citet{ye1998measuring} of simply taking $\alpha = 0.6$ does
deliver consistently favorable performance throughout.)

Overall, for small $\alpha$, the average BY estimate appears to be empirically 
close to $\Risk(g)$ in all situations (as does the average CB estimate), but we 
reiterate that there is no guarantee this will be true in general for nonsmooth
$g$ (as in panels (c) and (d)). However, the average CB estimate will always be
close to $\Risk_\alpha(g)$, which will be in turn close to $\Risk(g)$ for small
$\alpha$, no matter the smoothness of $g$ (Propositions
\ref{prop:risk_alpha_smoothness} and \ref{prop:cb_bias}).  

In the previous figure, the variability of the BY and CB estimates (reflected in
the standard error bars) appears roughly similar throughout.  In Figure
\ref{fig:risk_comparison2}, we demonstrate that this need not be the case in
general.  By increasing the true sparsity level to $s=200$ (the true linear
model is dense) and the signal-to-noise ratio to $\mathrm{SNR}=2$, we see that
the BY estimates appear much more volatile than those from CB, when we take $g$
to be the lasso tuned by cross-validation.  This holds across all values of
$\alpha$.  In Appendix \ref{app:additional_experiments}, we show that the  
larger variance of BY in this setting is due to its irreducible variance, and in
particular, just one part of its irreducible variance: comparing
\eqref{eq:cb_ivar} and \eqref{eq:by_ivar}, we see that the only difference
between the two is the first term (inside the variance).  In CB, this is the
conditional expectation of the noise-added training error, and in BY, it is the
training error itself.  When $g$ is unstable, as in the current setting (the use
of cross-validation for tuning induces instability into the ultimate prediction
function), the latter can be much more variable.

\subsection{Degrees of freedom}

Recalling Efron's covariance decomposition \eqref{eq:cov_decomp}, and the
definition of degrees of freedom \eqref{eq:df}, it is clear that estimating
$\Risk(g)$ and estimating $\df(g)$ are equivalent problems, in the normal means
setting.  Thus, parallel to the perspective and development used in this paper,
where the CB method \eqref{eq:cb_risk} is crafted as an unbiased estimator of
$\Risk_\alpha(g)$, the risk of $g$ at the elevated noise level of $(1+\alpha)
\sigma^2$, we can equivalently view:
\begin{equation}
\label{eq:cb_df}
\widehat{\df}_\alpha(g) = \frac{\CB_\alpha(g) - \frac{1}{B} \sum_{b=1}^B\| Y^\pb
  - g(Y^\pb) \|_2^2 + n\sigma^2(1+\alpha)}{2\sigma^2(1+\alpha)}
\end{equation}
as an unbiased estimator of $\df_\alpha(g)$, the degrees of freedom of $g$ at
the elevated noise level $(1+\alpha) \sigma^2$.  For the BY method,
meanwhile, one can proceed similarly in moving from \eqref{eq:by_risk} to an
estimator of degrees of freedom (by subtracting off training error and
rescaling); however, there is an alternative and more direct estimator that
stems from this method, which was the original proposal of
\citet{ye1998measuring}, namely: 
\begin{equation}
\label{eq:by_df}
\widetilde{\df}_\alpha(g) = \frac{1}{\sigma^2 \alpha} \sum_{i=1}^B \hCov_i^*,   
\end{equation}
where \smash{$\hCov_i^*$}, $i=1,\ldots,n$, are as in \eqref{eq:bootstrap_cov}.
While \smash{$\widehat{\df}_\alpha(g)$} estimates \smash{$\df_\alpha(g)$}
(unbiasedly), it seems that \smash{$\widetilde{\df}_\alpha(g)$} is designed to
directly estimate $\df(g)$ (though not unbiasedly). 

In Figure \ref{fig:df_comparison}, we evaluate the performance of these two
degrees of freedom estimators \eqref{eq:cb_df}, \eqref{eq:by_df} using the same
simulation framework as that described in the last subsection, with $s=5$ and
$\mathrm{SNR}=2$.  We consider two functions $g$: lasso and forward stepwise,
and for each, we vary their tuning parameters over their effective ranges.
Lastly, we fix $\alpha=0.1$.  The figure displays the estimated degrees of
freedom from CB \eqref{eq:cb_df} or BY \eqref{eq:by_df}, against the support
size of the underlying fitted sparse regression model (for the lasso, we take
this to be the average support size for the given value of $\lambda$ over all
100 repetitions): the bands represent the degrees of freedom estimate plus and
minus one standard error, over the 100 repetitions.  The true degrees of freedom
(itself estimated via Monte Carlo) is plotted as a dashed line. To be clear, this
plots the degrees of freedom $\df(g)$ at the original noise level, not the 
noise-elevated degrees of freedom $\df_\alpha(g)$.  We see that both methods
provide reasonably accurate estimates of $\df(g)$ throughout, albeit slightly
biased upwards at various points along the path (support sizes), due to the use
of $\alpha=0.1$.  Reducing $\alpha$ would reduce the bias, but also increase the
variability.  We also see that the estimates of degrees of freedom from the CB
method are just a bit more variable across the lasso path, and most noticeably
so at the smallest support sizes.   

\begin{figure}[htb]
\centering
\includegraphics[width=0.85\textwidth]{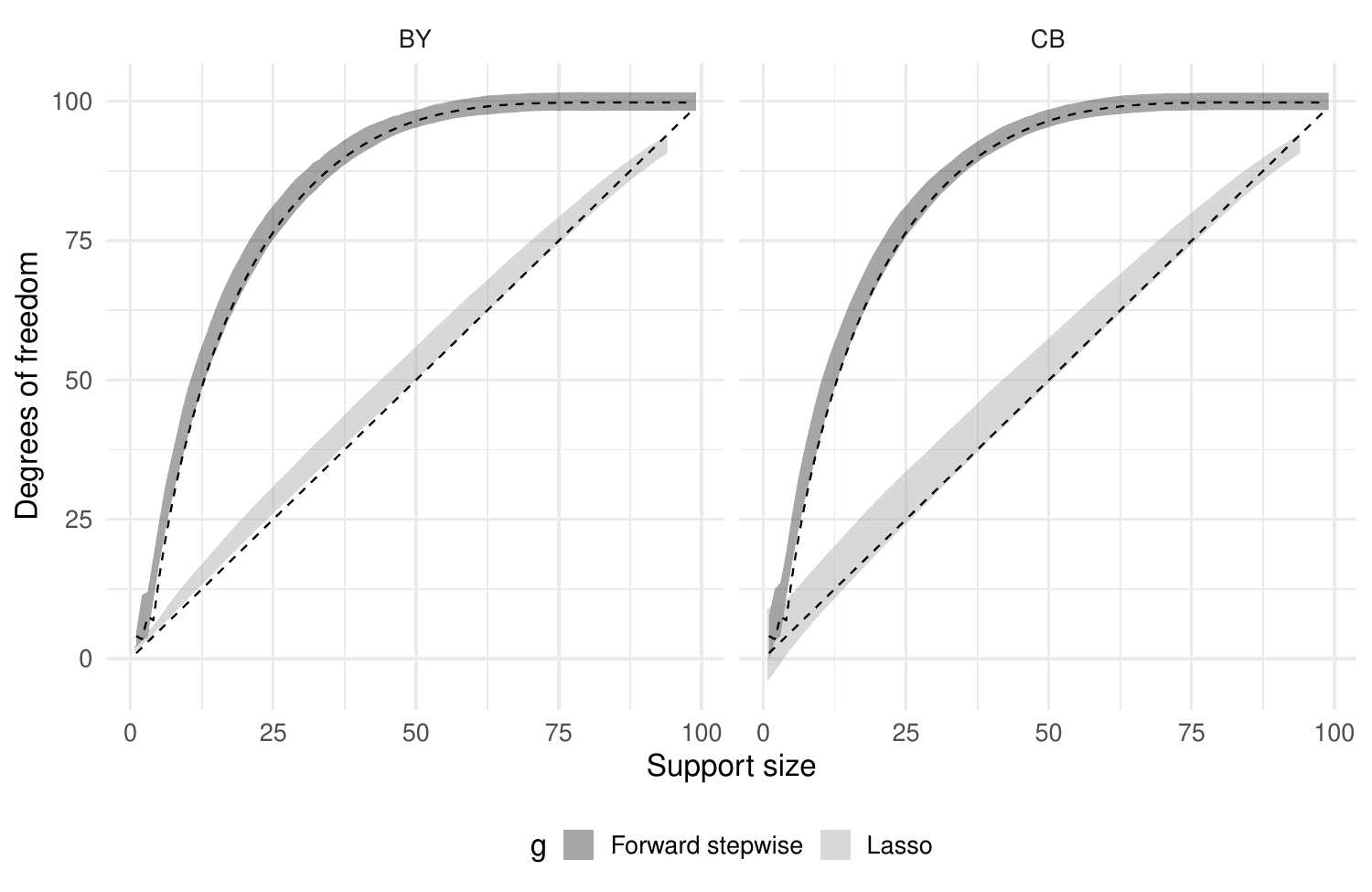}
\caption{Comparison of CB and BY degrees of freedom estimators applied to the  
  full forward stepwise and lasso paths, when $s=5$, $\mathrm{SNR}=2$, and
  $\alpha=0.1$.}  
\label{fig:df_comparison}
\end{figure}

\subsection{Image denoising}
\label{sec:image_denoising}

As a last example, we consider using the CB method for tuning parameter
selection in image denoising.  In image denoising, and signal processing more
broadly, SURE has become a central method for risk estimation and parameter
tuning (see Section \ref{sec:related_work} for references).  We focus on the
2-dimensional fused lasso \citep{tibshirani2005sparsity, hoefling2010path} as an
image denoising estimator, as it is weakly differentiable and its divergence can
be computed in analytic form \citep{tibshirani2011solution,
  tibshirani2012degrees}. This allows us to draw a comparison to SURE, which
takes the simple form:   
$$
\SURE(g) = \|Y - g(Y)\|_2^2 + 2\sigma^2 \big( \text{\# of fused groups in 
  $g(Y)$} \big) - n\sigma^2.
$$
To compare the CB estimator \eqref{eq:cb_risk} and SURE (above) empirically, we
use the standard ``parrot'' image from the image processing literature (leftmost
panel of Figure \ref{fig:parrot_images}), and we generate data $Y$ by adding
i.i.d.\ normal noise to each pixel (second from the left in Figure
\ref{fig:parrot_images}). Figure \ref{fig:parrot_risk} compares SURE (dashed
line) and CB (solid lines) across several values of $\alpha$, each as functions
of the underlying tuning parameter $\lambda$ in the 2d fused lasso optimization
problem.  The true risk is also plotted (dotted line).  The main conclusion is
that, for all values of $\alpha$ (even the largest, $\alpha=0.5$), the
minimizers of the CB curve over $\lambda$ are all close to that of SURE, which
means that the subsequent CB-tuned and SURE-tuned estimates are themselves all
quite similar (second to right and rightmost panels of Figure
\ref{fig:parrot_images}).  This speaks---informally---to model selection being
``easier'' than risk estimation in this context, as we can get away with larger
values of $\alpha$ and still make the relevant risk comparisons needed in order
to accurately select a model (indexed by a tuning parameter). In Appendix
\ref{app:additional_experiments}, we provide a more in-depth view by aggregating
model selection results in this image denoising simulation over multiple
repetitions.  

\begin{figure}[htb]
\centering
\includegraphics[width=0.7\textwidth]{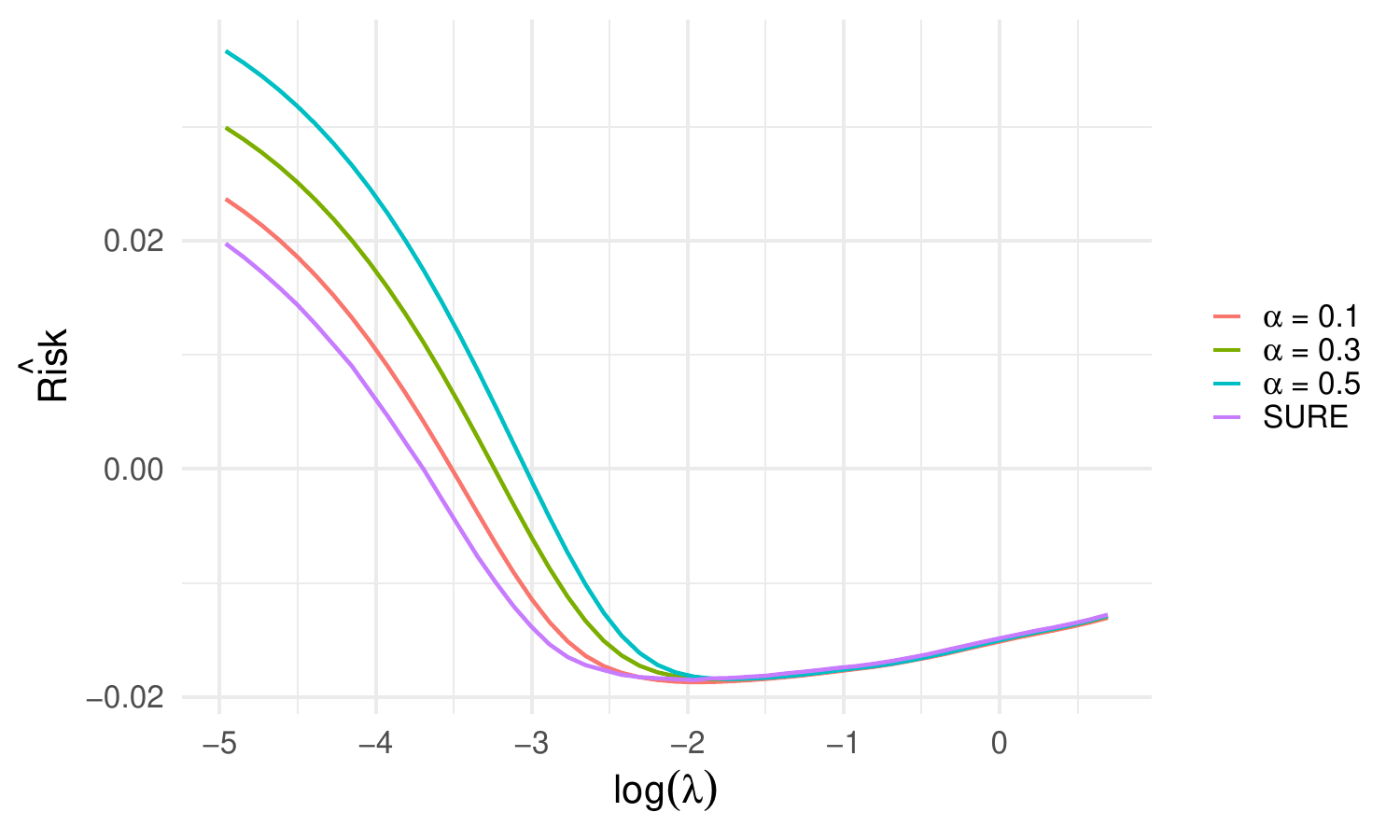}
\caption{Comparison of CB and SURE in image denoising.}       
\label{fig:parrot_risk}

\bigskip
\includegraphics[width=\textwidth]{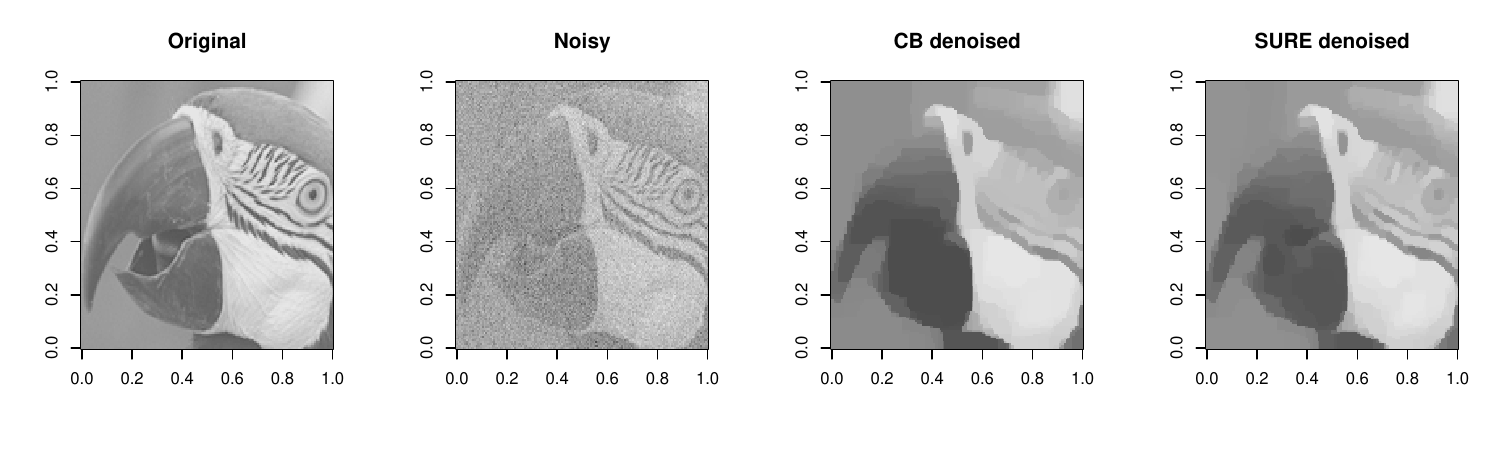} 
\caption{Original ``parrot'' image (leftmost), and a noisy version (second from  
  left) used for image denoising. The CB-tuned (second from right) and
  SURE-tuned (righmost) estimates look very similar. The CB tuning here uses 
  $\alpha=0.5$, which corresponds the biggest gap in the selected $\lambda$ to
  that from SURE.}   
\label{fig:parrot_images}
\end{figure}

\section{Discussion}
\label{sec:discussion}

In this work, we proposed and studied a coupled bootstrap (CB) method for risk
estimation in the standard normal means problem.  Our estimator is on one hand
similar to bootstrap-based proposals for risk estimation (via a covariance
decomposition) in this setting from \citet{breiman1992little, ye1998measuring, 
  efron2004estimation}.  On the other hand, it is different in a crucial way:
for any value of the auxiliary (bootstrap) noise parameter $\alpha>0$, the CB 
estimator is unbiased for $\Risk_\alpha(g)$, the risk of the given function $g$,
when the noise level in the normal means problem is elevated from
$\sigma^2$ to $(1+\alpha) \sigma^2$.  We proved that when $g$ is weakly
differentiable, the CB estimator (with infinite bootstrap iterations) reduces to 
SURE as $\alpha \to 0$. The same is true of the Breiman-Ye estimator does in
this noiseless limit. However, for nonsmooth $g$ and an arbitrary
non-infinitesimal $\alpha$, the CB estimator still tracks an intuitively
reasonable target: $\Risk_\alpha(g)$. Indeed, it is always unbiased for
$\Risk_\alpha(g)$, requiring no assumptions on $g$ whatsoever, which is a unique
property.  As such, it can be applied to arbitrarily complex functions $g$,
including those that use some sort of internal tuning parameter selection
mechanism.        


Of course, one of the most important practical problems not addressed in the
current paper is estimation of the error variance $\sigma^2$. In practice, the
simplest strategies here tend to be among the most commonly-used, and among the
most effective: we could simply estimate $\sigma^2$ by using the sample variance
of training residuals $Y_i - g_i(Y)$, $i=1,\ldots,n$ (possibly with a degrees of
freedom correction, which can be important for complex models $g$). A study of
the impact of estimating $\sigma^2$ on risk estimation and model selection is a
topic for future work. (For model selection, we may not actually need to
estimate $\sigma^2$ up to the same degree of accuracy as we would in risk
estimation; recall, we have already seen that relative large values of the
auxiliary noise level $\alpha > 0$ can still result in good model selection
performance, in Figure \ref{fig:parrot_risk}.)       

Two more extensions of the CB framework that may be of interest for future
work are described below.

\subsection{General error covariance}

Consider, instead of \eqref{eq:data_model}, data drawn according to:
\begin{equation}
\label{eq:data_model_sig}
Y \sim N(\theta, \Sigma),
\end{equation}
for a positive definite covariance matrix $\Sigma \in \R^{n \times n}$.  In such a
structured error setting, it may be of interest to 
measure loss according to a generalized quadratic norm, thus we introduce the 
notation $\|x\|_A^2 = x^\T A^{-1} x$ for a vector $x$ and positive semidefinite
matrix $A$.  For example, we may choose to measure loss according to
\smash{$\|\theta - g(Y)\|_\Sigma^2$}, since the curvature in this loss takes
$\Sigma$ into account, just like the negative log-likelihood in the model
\eqref{eq:data_model_sig}.  

We extend the CB estimator so that it applies to an arbitrary positive
semidefinite matrix $A$ defining the risk, and an arbitrary positive
semidefinite matrix $\Sigma$ in \eqref{eq:data_model_sig}.  The next result is a
straightforward extension of Proposition \ref{prop:three_point}.

\begin{proposition}
\label{prop:three_point_a}
Let $U,V,W \in \R^n$ be independent random vectors.  Then for any $g$, and  
positive semidefinite matrix $A \in \R^{n \times n}$, 
\begin{equation}
\label{eq:three_point_a1}
\E \|V - g(U)\|_A^2 - \E \|W - g(U)\|_A^2 = \E\|V\|_A^2 - \E\|W\|_A^2 + 2 
\langle A^{-1} \E[g(U)], \E[W] - \E[V] \rangle. 
\end{equation}
assuming all expectations exist and are finite. In particular, if $U,V$ are
i.i.d.\ and $\E[U] = \E[W]$, then
\begin{equation}
\label{eq:three_point_a2}
 \E\|V - g(U)\|_A^2 = \E \|W - g(U)\|_A^2 + \E\|U\|_A^2 - \E\|W\|_A^2. 
\end{equation}
\end{proposition}

And in turn, the next result is a straightforward extension of Corollary
\ref{cor:cb_unbiased}.

\begin{corollary}
\label{cor:cb_unbiased_sig}
Let $Y \sim N(\theta, \Sigma)$.  Given any function $g$, a positive semidefinite 
matrix $A \in \R^{n \times n}$ that will be used to measure risk, and an
auxiliary noise level $\alpha > 0$, consider defining a CB estimator according
to:   
\begin{equation}
\label{eq:cb_bootstrap_sig}
\begin{gathered}
\omega^b \sim N(0, \Sigma), 
\quad \text{independently}, \quad \text{for $b=1,\ldots,B$}, \\
Y^\pb = Y + \sqrt\alpha \omega^b, \quad 
Y^\mb = Y - \omega^b / \sqrt\alpha, \quad \text{for $b=1,\ldots,B$},  
\end{gathered}
\end{equation}
and:
\begin{equation}
\label{eq:cb_risk_sig}
\CB_{A,\alpha}(g) = \frac{1}{B} \sum_{b=1}^B \Big(\|Y^\mb - g(Y^\pb)\|^2_A - 
\|\omega^b\|_A^2 / \alpha\Big) - \tr(A^{-1} \Sigma).
\end{equation}
Then this is unbiased for risk at the noise-elevated level $(1+\alpha) \Sigma$ 
measured with respect to $A$, i.e.,
$$
\E[\CB_{A,\alpha}(g)] = \Risk_{A,\alpha}(g) = \E\|\theta - g(Y_\alpha)\|_A^2, 
\quad \text{where $Y_\alpha \sim N(\theta, (1+\alpha) \Sigma)$}.
$$
\end{corollary}

Of course, the main challenge in using the extended estimator
\smash{$\CB_{A,\alpha}(g)$} defined in the above corollary is that it requires
knowledge of the full error covariance matrix $\Sigma$.  However, in some
settings, e.g., time series problems, it may be reasonable to assume that
$\Sigma$ or its inverse is highly structured and therefore estimable.  It may be
interesting to rigorously study how risk estimation is affected by upstream
estimation of $\Sigma$ in this and related problem settings. 

\subsection{Bregman divergence}

Finally, we present a further extension of the simple and yet key results in 
Proposition \ref{prop:three_point_a} underpinning the construction of the CB
estimator, to the case in which a Bregman divergence is used to measure error: 
\begin{equation}
\label{eq:err_bregman}
\Err(g) = \E[D_\phi(\tilde{Y}, g(Y))], \quad \text{where $\tilde{Y}$ is an
  i.i.d.\ copy of $Y$}. 
\end{equation}
Here $D_\phi$ is the \emph{Bregman divergence} with respect to a strictly
convex and differentiable function $\phi : \R^n \to \R$, which recall is defined 
by: 
$$
D_\phi(a,b) = \phi(a) - \phi(b) - \langle \nabla \phi(b), a-b \rangle. 
$$
When $\phi(x) = \|x\|_2^2$, it is easy to check that 
$$
D_{\|\cdot\|_2^2}(a,b) = \|a\|_2^2 - \|b\|_2^2 - 2 \langle b, a-b \rangle = \|a
- b\|_2^2, 
$$
and hence \eqref{eq:err_bregman} reduces to prediction error as measured by 
squared loss in \eqref{eq:pred_error}.  In fact, properties
\eqref{eq:three_point_a1}, \eqref{eq:three_point_a2} are entirely driven by this
``Bregman representation'' of squared loss, leading to the following extension.    

\begin{proposition}
\label{prop:three_point_breg}
Let $U,V,W \in \R^n$ be independent random vectors.  For any $g$, and 
Bregman divergence $D_\phi$, 
\begin{equation}
\label{eq:three_point_breg1}
\E[ D_\phi(V, g(U)) ] - \E[ D_\phi(W, g(U)) ] = \E[\phi(V)] - \E[\phi(W)] +
\langle \E[\nabla\phi(U)], \E[W] - \E[V] \rangle.   
\end{equation}
assuming all expectations exist and are finite. In particular, if $U,V$ are
i.i.d.\ and $\E[U] = \E[W]$, then
\begin{equation}
\label{eq:three_point_breg2}
\E[ D_\phi(V, g(U)) ]  = \E[ D_\phi(W, g(U)) ]  + \E[\phi(U)] - \E[\phi(W)] . 
\end{equation}
\end{proposition}

Proposition \ref{prop:three_point_breg} is, in principle, a powerful tool: it
provides ``one half'' of a recipe to move the CB estimator beyond the Gaussian 
setting, to a setting in which data follows (say) an exponential family
distribution and loss is measured by the out-of-sample deviance.  This is
because for every exponential family distribution, there is a natural function
$\phi$ (defined in terms of the log-partition function of the distribution) that
makes \eqref{eq:err_bregman} the deviance.  

The ``other half'' of the recipe needed to arrive at a CB estimator is a
mechanism for generating relevant bootstrap draws, as in
\eqref{eq:cb_bootstrap_sig} in the previous subsection.  Specifically, for a
given problem setting with data $Y$  (exponential family distributed or
otherwise) we must be able to design a pair of bootstrap draws \smash{$(Y^\pb, 
  Y^\mb)$} that adhere to three criteria: 
\begin{enumerate}
\item $Y^\pb, Y^\mb$ are independent of each other;
\item $\E[Y^\pb] = \E[Y^\mb]$; and
\item $\E[D_\phi(\tilde{Y}^\pb, g(Y^\pb))]$ is an ``interesting'' pseudo-target
  to estimate, where $\tilde{Y}^\pb$ is an i.i.d.\ copy of $Y^\pb$. 
\end{enumerate}
Criteria 1 and 2 are straightforward enough to understand, and they should be
possible to fulfill in certain exponential family models with various noise
augmentation tricks.  However, criterion 3 deserves a bit more explanation.
With \smash{$U=Y^\pb$}, \smash{$V=\tilde{Y}^\pb$}, and \smash{$W=Y^\mb$},
assumed to fulfill criteria 1 and 2, note that \eqref{eq:three_point_breg2}
says \smash{$D_\phi(Y^\mb, g(Y^\pb))$} is unbiased for
\smash{$\E[D_\phi(\tilde{Y}^\pb, g(Y^\pb))]$}.  That is, we originally wanted to
estimate the quantity in \eqref{eq:err_bregman}, and have now pivoted to
estimating \smash{$\E[D_\phi(\tilde{Y}^\pb, g(Y^\pb))]$} instead.   

In the Gaussian setting studied in this paper, this meant estimating
risk based on data from a Gaussian distribution with the same mean but an
elevated noise level.  In a more general setting, the noise augmentation
strategy used to generate \smash{$Y^\pb$} may in fact bring us outside of the
distributional family assumed for the data originally, and it may even alter
non-nuisance parameters of the distribution. This would still altogether be
fine, as long as \smash{$\E[D_\phi(\tilde{Y}^\pb, g(Y^\pb))]$} it still an
``interesting'' target (i.e., for error assessment or model selection), as per 
criterion 3.  

As an a concrete example of where such an extension is possible (precisely along
the lines of the above discussion), we note that after the completion of the
current paper, we were able to extend the CB framework to the Poisson many means
model, in \citet{oliveira2022unbiased}.      

\subsection*{Acknowledgements} 

RJT is grateful to Xiaoying Tian for providing the inspiration to work on this
project in the first place, and Saharon Rosset for early insightful 
conversations.  NLO was supported by an Amazon Fellowship.  

\bibliographystyle{plainnat}
\bibliography{ryantibs}

\newpage
\appendix

\section{More details on Breiman's and Ye's estimators}
\label{app:by_estimators}

Instead of defining \smash{$\hCov_i^*$}, $i=1,\ldots,n$ as in
\eqref{eq:bootstrap_cov}, Breiman uses 
$$
\hCov_i^*= \frac{1}{B-1} \sum_{b=1}^B (Y_i^{*b} - Y_i) g_i(Y^{*b}), \quad
i=1,\ldots,n,
$$
which are just inner products between the noise increments \smash{$\{Y_i^{*b}
  - Y_i\}_{b=1}^B$} and fitted values \smash{$\{g_i(Y^{*b})\}_{b=1}^B$}, instead 
of an empirical covariances.   

Furthermore, instead of dividing the whole sum by $\alpha$, Ye divides each
summand \smash{$\hCov_i^*$} in \eqref{eq:by_risk} by 
$$
(s^*_i)^2 = \frac{1}{B-1} \sum_{b=1}^B (Y_i^{*b} - \bar{Y}_i^*)^2, 
$$
the bootstrap estimate of the variance of $Y_i$, rather than dividing the entire
sum by $\alpha$.  In fact, Ye actually formulates his estimator in terms of the
slopes from linearly regressing the fitted values
\smash{$\{g_i(Y^{*b})\}_{b=1}^B$} onto the noise increments \smash{$\{Y_i^{*b} -
  Y_i\}_{b=1}^B$}, but it is equivalent to the form described here. 

\section{Proof of Proposition \ref{prop:risk_alpha_smoothness}} 
\label{app:risk_alpha_smoothness}

The proposition follows from an application of the next lemma, as we can take
$f(y) = \|\theta - g(y)\|_2^2$, and then the moment conditions on $f$ will be
implied by those on $\|g\|_2^2$, via the simple bound $f(y) \leq 2
\|\theta\|_2^2 + 2 \|g(y)\|_2^2$.   

\begin{lemma}
\label{lem:f_alpha_smoothness}
For $\alpha \geq 0$, denote $Y_\alpha \sim N(\theta, (1+\alpha) \sigma^2 I_n)$. 
Let $f : \R^n \to \R$ be a function such that, for some $\beta > 0$ and
integer $k \geq 0$,    
$$
\E \big[ f(Y_\beta) \|Y_\beta - \theta\|_2^{2m} \big] < \infty, 
\quad m=0,\ldots,k. 
$$
Then, the map $\alpha \mapsto \E[f(Y_\alpha)]$ has $k$ continuous derivatives on
$[0, \beta)$.   
\end{lemma}

\begin{proof}
First, we prove that this map is continuous.  Fix $\alpha \in [0, \beta)$.
Observe that 
\begin{align*}
\lim_{t \to \alpha} \E[f(Y_t)] 
&= \lim_{t \to \alpha} \int \frac{f(y)}{(2\pi(1+t)\sigma^2)^{n/2}}  
\exp\bigg\{\frac{-\|y - \theta\|^2}{2(1+t)\sigma^2} \bigg\} \, dy \\ 
&= \int \lim_{t \to \alpha} \frac{f(y)}{(2\pi(1+t)\sigma^2)^{n/2}}  
\exp\bigg\{\frac{-\|y - \theta\|^2}{2(1+t)\sigma^2} \bigg\} \, dy \\
&=  \E[f(Y_\alpha)],
\end{align*}
where in the second line we used Lebesgue's dominated convergence theorem (DCT), 
applicable because the integrand is bounded by 
$$
\frac{f(y)}{(2\pi\sigma^2)^{n/2}} 
\exp\bigg\{\frac{-\|y - \theta\|^2}{2(1+\alpha)\sigma^2} \bigg\}, 
$$
which is integrable by assumption.  Now for the first derivative, note that  
$$
\frac{\partial}{\partial \alpha} \E[f(Y_\alpha)] = 
-\frac{n}{2(1+\alpha)} \E[f(Y_\alpha)] 
+\frac{1}{2\sigma^2(1+\alpha)^2} \E[f(Y_\alpha) \|Y_\alpha - \theta\|_2^2], 
$$
where we used the Leibniz integral rule, applicable because the integrands 
(when we write these expectations as integrals) are bounded by 
$$
\frac{f(y)}{(2\pi\sigma^2)^{n/2}} \|y - \theta\|_2^{2m}
\exp\bigg\{\frac{-\|y - \theta\|^2}{2(1+\alpha)\sigma^2} \bigg\}, 
$$
for $m=0,1$, again integrable by assumption.  Another application of DCT proves
the derivative in the second to last display is continuous on $[0, \beta)$.  For
a general number of derivatives $k$, the argument is similar, and the
integrability of the dominating functions in the above display, for
$m=0,\ldots,k$, ensures that we can apply the Leibniz rule and DCT to argue
continuity of the $k$th derivative on $[0,\beta)$.   
\end{proof}

\section{Proof of Theorem \ref{thm:cb_noiseless}}
\label{app:cb_noiseless}

\subsection{Proof of theorem}

Observe that, writing $\E_\omega$ for the conditional expectation operator on
$Y=y$ (i.e., the operator that integrates over $\omega$), 
\begin{align*}
\CB_\alpha^\infty(g) &= \E_\omega \big[ \|y - \omega/\sqrt\alpha - g(y +
\sqrt\alpha \omega)\|_2^2 -  \|\omega\|_2^2 / \alpha \big] - n\sigma^2 \\  
&= \underbrace{\vphantom{\frac{2}{\sqrt\alpha}} \E_\omega \|y - g(y +
\sqrt\alpha \omega)\|_2^2}_{a} \,-\,  
\underbrace{\frac{2}{\sqrt\alpha} \E_\omega \langle \omega, g(y + \sqrt\alpha
\omega) \rangle}_{b} - n\sigma^2. 
\end{align*}
It is not hard to show that for almost every $y \in \R^n$, it holds that $a \to
\|y - g(y)\|_2^2$ as $\alpha \to 0$, by Lemma \ref{lem:lebesgue_point}.  It
remains to study term $b$.   

Denote by \smash{$\phi_{\mu, \sigma^2}$} the density of a Gaussian with mean 
$\mu$ and variance $\sigma^2$.  Then,  
\begin{align*}
b &= \frac{2}{\sqrt{\alpha}} \sum_{i=1}^n \E_{\omega_{-i}}
\E_{\omega_i} [ \omega_i g_i(y + \sqrt\alpha \omega) ] \\
&= \frac{2}{\sqrt{\alpha}} \sum_{i=1}^n \E_{\omega_{-i}}
\int \omega_i g_i(y + \sqrt\alpha \omega) \phi_{0, \sigma^2}(\omega_i) 
\, d\omega_i \\
&= -\frac{2\sigma^2}{\sqrt{\alpha}} \sum_{i=1}^n \E_{\omega_{-i}} 
\int g_i(y + \sqrt\alpha \omega) \phi'_{0, \sigma^2}(\omega_i) \, d\omega_i \\ 
&= -2\sigma^2 \sum_{i=1}^n \E_{u_{-i}} 
\int g_i(u) \phi'_{0, \alpha\sigma^2}(u_i - y_i) \, du_i \\ 
&= 2\sigma^2 \sum_{i=1}^n \E_{u_{-i}} 
\int \nabla_i g_i(u) \phi_{0, \alpha\sigma^2}(u_i - y_i) \, du_i \\ 
&= 2\sigma^2 \sum_{i=1}^n \int \nabla_i g_i(u) \phi_{0, \alpha\sigma^2}(u - y)  
\, du_i. 
\end{align*}
The second to last line holds by Lemma \ref{lem:int_by_parts}.  Now, by Lemma 
\ref{lem:lebesgue_point}, for almost every $y \in \R^n$,
$$
\lim_{\alpha \to 0} 2\sigma^2 \sum_{i=1}^n \int \nabla_i g_i(u) \phi_{0,
\alpha\sigma^2}(u - y) \, du = 2\sigma^2 \sum_{i=1}^n \nabla_i g_i(y),
$$
which completes the proof.

\subsection{Supporting lemmas}

Here we state and prove supporting lemmas for the proof of Theorem
\ref{thm:cb_noiseless}.  The first lemma shows that for a weakly differentiable
function, the integration by parts property in \eqref{eq:weak_diff} still holds 
when we take the test function to be a normal density (which is continuously
differentiable by not compactly supported).

\begin{lemma}
\label{lem:int_by_parts}
If $f : \R \to \R$ is weakly differentiable, with \smash{$(f \phi_{\mu, 
    \sigma^2}) \in L^1(\R)$} and \smash{$(f' \phi_{\mu, \sigma^2}) \in
  L^1(\R)$}, then  
$$
\int f(x) \phi'_{\mu, \sigma^2}(x) \, dx = -\int f'(x) \phi_{\mu, \sigma^2}(x)
\, dx. 
$$
\end{lemma}

\begin{proof}
Let $\psi_n: \R \to [0,1]$, $n=1,2,3,\ldots$ be a sequence of continuously
differentiable functions such that for each $z \in \R$,
$$
\lim_{n \to \infty} \psi_n(z) = 1, \quad
\lim_{n \to \infty} \psi_n'(z) = 0, \quad \text{and} \quad 
|\psi_n'(z)| \leq C \;\, \text{for $n=1,2,3,\ldots$ and a constant $C<\infty$}.  
$$
One example of such a sequence is 
$$
\psi_n(z) = 1_{(-n,n)}(z) + \exp\bigg( -\frac{1}{1-(z - n\sign(z))} \bigg)
1_{[-n-1,-n]\cup[n, n+1]}(z), \quad n=1,2,3,\ldots.
$$
Now let \smash{$\xi_n(z) = \psi_n(z)\phi_{\mu, \sigma^2}(z)$}.  Note that
\begin{align*}
\lim_{n \to \infty} \xi_n(z) &= \phi_{\mu, \sigma^2}(z) \lim_{n \to \infty}
\psi_n(z) = \phi_{\mu, \sigma^2}(z), \\
\lim_{n \to \infty} \xi'_n(z) &= \lim_{n \to \infty} \psi'_n(z) \phi_{\mu,
\sigma^2}(z) + \phi_{\mu, \sigma^2}'(z) \lim_{n \to \infty}\psi_n(z) =
\phi_{\mu, \sigma^2}'(z).   
\end{align*}
Turning to the result we want to prove,
\begin{align*}
\int f(z) \phi_{\mu, \sigma^2}'(z) \, dz 
&= \int f(z) \lim_{n \to \infty} \xi'_n(z) \, dz \\
&= \lim_{n \to \infty} \int f(z) \xi'_n(z) \, dz \\
&= -\lim_{n \to \infty} \int f'(z) \xi_n(z) \, dz \\
&= -\int f'(z) \lim_{n \to \infty} \xi_n(z) \, dz \\
&= -\int f'(z) \phi_{\mu, \sigma^2}(z) \, dz.
\end{align*}
The second and fourth lines here can be verified using Lebesgue's dominated
convergence theorem (DCT), and the third uses \eqref{eq:weak_diff}, applicable
because each $\xi_n$ is compactly supported.  This completes the proof. 
\end{proof}

The next lemma essentially shows that the notion of a Lebesgue point can be 
extended to the Gaussian kernel (beyond the uniform kernel, as it is
traditionally defined). 

\begin{lemma}[Adapted from Theorem 1.25 of \citealt{stein1971introduction}] 
\label{lem:lebesgue_point}
Let $\phi : \R^n \to \R$ be the Gaussian density with mean zero and identity
covariance, and denote $\phi_\alpha = \alpha^{-n} \phi(x/\alpha)$.  Let $f :
\R^n \to \R$ be a function such that $(f \phi_\beta) \in L^1(\R^n)$ for some
$\beta>0$. Then, $\lim_{\alpha \to 0} (f * \phi_\alpha)(x) = f(x)$ for almost
every $x \in \R^n$.  
\end{lemma}

\begin{proof}
Let $x \in \R^n$ be a Lebesgue point of $f$.  We will prove that the desired
result holds for $x$, which will imply that it holds almost everywhere (because 
any function in $L^1_{\mathrm{loc}}(\R^n)$ has the property that almost every
point is a Lebesgue point; see, e.g., Theorem 1.32 of \citet{evans2015measure}).

Fix $\epsilon > 0$.  By the definition of a Lebesgue point, there exists $\rho > 
0$ such that  
\begin{equation}
\label{eq:lebesgue_point1}
\delta^{-n} \int_{\|t\|_2 \leq \delta} |f(x-t) - f(x)| \, dt \leq C \epsilon, 
\end{equation}
for all $\delta \in (0, \rho]$ and a constant $C>0$ to be specified later.  In
what follows, we will show that there exists $\beta > 0$ such that $|(f *  
\phi_\alpha)(x) - f(x)|\leq \epsilon$ for all $\alpha \in (0, \beta]$.  To do
so, we decompose
\begin{equation}
\label{eq:i1_i2}
|(f * \phi_\alpha)(x) - f(x)| \leq 
\underbrace{\bigg| \int_{\|t\|_2 \leq \delta} (f(x-t)-f(x)) \phi_\alpha(t) \, dt  
  \bigg|}_{I_1} \,+\, 
\underbrace{\bigg| \int_{\|t\|_2 > \delta} (f(x-t)-f(x)) \phi_\alpha(t) \, dt
  \bigg|}_{I_2}. 
\end{equation}
We study each term above separately.

\paragraph{Term $I_1$.} 

Let \smash{$g(r) = \int_{\|t\|_2=1} |f(x-rt) - f(x)| \, dt$} and $G(r) =
\int_0^r s^{n-1}g(s) \, ds$.  Note that \eqref{eq:lebesgue_point1} translates 
into the statement 
\begin{equation}
\label{eq:lebesgue_point2}
G(r) \leq C \epsilon r^n,
\end{equation}
for all $r \in (0, \rho]$.   

For notational convenience, we write $\varphi(r) = \phi(u)$ whenever $\|u\|_2 =
r$, where $\varphi$ is the univariate standard normal density.  Observe that for
any $\delta \leq \rho$,   
\begin{align*}
I_1 &\leq \int_{\|t\|_2 \leq \delta} |f(x-t)-f(x)| \alpha^{-n} \phi(t/\alpha) \,
dt \\
&= \int_0^\delta r^{n-1} g(r) \alpha^{-n} \varphi(r/\alpha) \, dr \\
&= G(r) \alpha^{-n} \varphi(r/\alpha) \big|_0^\delta -
\alpha^{-n} \int_0^\delta G(r) \, d(\varphi(r/\alpha)) \\
&\leq C \epsilon (\delta/\alpha)^n \varphi(\delta/\alpha) - 
\alpha^{-n} \int_0^\delta G(r) \, d(\varphi(r/\alpha)) \\  
&= C \epsilon (\delta/\alpha)^n  \varphi(\delta/\alpha) -  
\alpha^{-n} \int_0^{\delta/\alpha} G(\alpha s) \, d(\varphi(s)) \\
&\leq C \epsilon (\delta/\alpha)^n \varphi(\delta/\alpha) +
C \epsilon \int_0^{\delta/\alpha} s^n \, |d(\varphi(s))| \\  
&\leq C \epsilon ( c_n + m_{n+1} ). 
\end{align*}
In the fourth and sixth lines, we used \eqref{eq:lebesgue_point2}.  In the last
line, we used the fact the map $z \mapsto z^n \varphi(z)$ attains a maximum  
of \smash{$c_n = \sqrt{n}^n \phi(\sqrt{n})$} at \smash{$z=\sqrt{n}$}, as well as
the bound \smash{$\int_0^{\delta/\alpha} s^n \, |d(\varphi(s))| \leq
  \int_0^\infty s^{n+1} \varphi(s) \, ds \leq m_{n+1}$}, where $m_{n+1}$
uncentered, absolute moment of order $n+1$ of the standard normal distribution.  
By choosing $C \leq 1/(2(c_n + m_{n+1}))$, we see that $I_1 \leq \epsilon/2$ for
any $\delta \leq \rho$, and any $\alpha > 0$. 

\paragraph{Term $I_2$.}  
Consider 
$$
I_2 = \underbrace{\int_{\|t\|_2 \geq \delta} |f(x-t)| \phi_\alpha(t) \,
  dt}_{I_{21}} \,+\, \underbrace{|f(x)| \int_{\|t\|_2 \geq \delta}
  \phi_\alpha(t) \, dt}_{I_{22}}. 
$$
Clearly
$$
\lim_{\alpha \to 0} I_{22} = |f(x)| \lim_{\alpha \to 0} \int_{\|u\|_2 \geq 
  \delta/\alpha} \phi(u) \, du = 0,
$$
so there exists $\beta_1>0$ such that for $\alpha \leq \beta_1$, we have $I_{22} 
\leq \epsilon/4$.  As for $I_{21}$, we have
$$
\lim_{\alpha \to 0} I_{21} = \lim_{\alpha \to 0} \int_{\|t\|_2 \geq \delta}
|f(x-t)| \phi_\alpha(t) \, dt = \int_{\|t\|_2 \geq \delta} \lim_{\alpha \to 0}
|f(x-t)| \phi_\alpha(t) \, dt = 0,
$$
where the interchange between integration and the limit as $\alpha \to 0$ can be 
shown using DCT. Thus there is $\beta_2>0$ such that for $\alpha \leq \beta_2$,
we have $I_{21} < \epsilon/4$. 

\paragraph{Completing the proof.}  

Putting the above parts together, we get that $I_1+I_2 \leq \epsilon/2 +
\epsilon/4 + \epsilon / 4 = \epsilon$, for all $\delta \leq \rho$ and $\alpha
\leq \beta = \min\{\beta_1,\beta_2\}$.  Recalling \eqref{eq:i1_i2}, this gives
the desired result and completes the proof.
\end{proof}

\section{Noiseless limit for hard-thresholding}
\label{app:cb_noiseless_ht}

The limit in question is that of 
$$
\frac{2}{\sqrt\alpha} \sum_{i=1}^n \E\big[ \omega_i (y_i + \sqrt\alpha \omega_i)  
\cdot 1\{|y_i + \sqrt\alpha \omega_i| > t\} \big]
$$
as $\alpha \to 0$.  Inspecting term $i$, 
\begin{align*}
\E\big[ \omega_i (y_i + \sqrt\alpha \omega_i) \cdot 1\{|y_i + \sqrt\alpha
\omega_i| > t\} \big] &= \frac{y}{\sqrt\alpha} \Bigg( \E\bigg[ \omega_i \cdot 
1\bigg\{ \omega_i \leq -\frac{t+y_i}{\sqrt\alpha} \bigg\} \bigg] +  
\E\bigg[\omega_i \cdot 1\bigg\{ \omega_i \geq \frac{t-y_i}{\sqrt\alpha} \bigg\} 
\bigg] \Bigg) +{} \\ 
&\qquad \E\bigg[ \omega_i^2 \cdot 1\bigg\{ \omega_i \leq -\frac{t+y_i}
{\sqrt\alpha} \bigg\} \bigg] +  \E\bigg[ \omega_i^2 \cdot 
1\bigg\{ \omega_i \geq \frac{t-y_i}{\sqrt\alpha} \bigg\} \bigg]. 
\end{align*}
To compute the above, we recall the identities, for $Z \sim N(0, \tau^2)$, 
\begin{align*}
\E\big[Z \cdot 1\{Z \leq a\}\big] &= -\tau \phi(a/\tau), \\
\E\big[Z \cdot 1\{Z \geq b\}\big] &= \tau \phi(b/\tau), \\
\E\big[Z^2 \cdot 1\{Z \leq a\}\big] &= -\tau a \phi(a/\tau) + 
\tau^2 \Phi(a/\tau), \\
\E\big[Z^2 \cdot 1\{Z \geq b\}\big] &= \tau b \phi(b/\tau) +
\tau^2 \bar\Phi(b/\tau),
\end{align*}
where $\phi$ and $\Phi$ denote the standard normal density and distribution
function, respectively, and \smash{$\bar\Phi = 1-\Phi$} the standard normal
survival function.  Thus we find that the second to last display equals  
\begin{align*}
\E\big[ \omega_i (y_i + \sqrt\alpha \omega_i) \cdot
&1\{|y_i + \sqrt\alpha \omega_i| > t\} \big] \\
&= \frac{\sigma y_i}{\sqrt\alpha} 
\bigg[ -\phi\bigg(\frac{t+y_i}{\sqrt\alpha \sigma} \bigg) +
\phi\bigg(\frac{t-y_i}{\sqrt\alpha \sigma} \bigg) \bigg] + 
\frac{\sigma}{\sqrt\alpha}
\bigg[ (t+y_i) \phi\bigg(\frac{t+y_i}{\sqrt\alpha \sigma} \bigg) +
(t-y_i) \phi\bigg(\frac{t-y_i}{\sqrt\alpha \sigma} \bigg)\bigg] +{} \\
&\qquad \sigma^2 \bigg[ \Phi\bigg(\frac{-t-y_i}{\sqrt\alpha \sigma} \bigg)  
+ \bar\Phi\bigg(\frac{t-y_i}{\sqrt\alpha \sigma} \bigg) \bigg] \\
&= \frac{\sigma t}{\sqrt\alpha} 
\bigg[ \phi\bigg(\frac{t+y_i}{\sqrt\alpha \sigma} \bigg) +
\phi\bigg(\frac{t-y_i}{\sqrt\alpha \sigma} \bigg) \bigg] + 
\sigma^2 \bigg[ \Phi\bigg(\frac{-y_i-t}{\sqrt\alpha \sigma} \bigg) 
+ \Phi\bigg(\frac{y_i-t}{\sqrt\alpha \sigma} \bigg) \bigg] \\
&\to \sigma^2 1\{|y_i| > t\}, \quad \text{for $y_i \not = \pm t$},
\vphantom{\bigg[ \bigg]}
\end{align*}
where the last line is the limit as $\alpha \to 0$.  In other words, we have
shown 
$$
\lim_{\alpha \to 0} \frac{2}{\sqrt\alpha} \sum_{i=1}^n \E\big[ \omega_i (y_i + 
\sqrt\alpha \omega_i) \cdot 1\{|y_i + \sqrt\alpha \omega_i| > t\} \big] 
= 2 \sigma^2 \sum_{i=1}^n 1\{|y_i| > t\} \quad \text{for $y_i \not = \pm t$,  
  $i=1,\ldots,n$}, 
$$
which proves \eqref{eq:cb_noiseless_ht}.

\section{Proofs of bias and variance results}
\label{app:bias_var}

\subsection{Proof of Proposition \ref{prop:cb_bias}}

Under the given assumptions on $g$, the map $\alpha \to \Risk_\alpha(g)$ is
continuously differentiable, and as shown in the proof of Proposition
\ref{prop:risk_alpha_smoothness}, we can use the Leibniz integral rule, to 
compute for $t \in [0,\alpha)$,
\begin{align*}
\frac{\partial}{\partial t} \Risk_t(g) &= 
\frac{1}{2(1+t)} \E \bigg[\|\theta - g(Y_t)\|_2^2 
\bigg(\frac{\|Y_t - \theta\|_2^2}{\sigma^2(1+t)} -  n\bigg)\bigg] \\  
&= \frac{1}{2(1+t)} \Cov\bigg(\|\theta - g(Y_t)\|_2^2, 
\frac{\|Y_t - \theta\|_2^2}{\sigma^2(1+t)}\bigg) \\
&= \frac{\sqrt{n}}{\sqrt{2}(1+t)} 
\sqrt{\Var(\|\theta - g(Y_t)\|_2^2)} \, 
\Cor\big(\|\theta - g(Y_t)\|_2^2, \|Y_t - \theta\|_2^2\big),
\end{align*}
where in the second line we used the fact that $\|Y_t - \theta\|_2^2 /
(\sigma^2(1+t)) \sim \chi_n^2$ and thus has mean $n$, and in the third
line we used that its variance is $2n$.  Applying the fundamental theorem of 
calculus
\[
\Risk_\alpha(g) - \Risk(g) = \int_0^\alpha \frac{\partial}{\partial t}
\Risk_t(g) \, dt 
\]
gives the result in \eqref{eq:cb_bias}.  The bound in \eqref{eq:cb_bias_bd1} is
obtained by bounding the correlation (between $\|\theta - g(Y_t)\|_2^2$ and
$\|Y_t - \theta\|_2^2$) by 1, and then using the assumed monotonicity of the
resulting integrand.     

For the second bound, in \eqref{eq:cb_bias_bd2}, observe that under the
additional (higher-order) moment conditions on $g$, the map $\alpha \mapsto
\Var(\|\theta - g(Y_\alpha)\|_2^2)$ is continuously differentiable on $[0,
\beta)$ by an application of Lemma \ref{lem:f_alpha_smoothness}. Thus we get  
$\Var(\|\theta - g(Y_\alpha)\|_2^2) = \Var(\|\theta - g(Y)\|_2^2) + O(\alpha)$
(say, by the fundamental theorem of calculus), which, along with the simple 
inequality \smash{$\sqrt{a+b} \leq \sqrt{a} + \sqrt{b}$} for $a,b \geq 0$, gives
the desired result.       

\subsection{Proof of Proposition \ref{prop:cb_rvar}}

Let $\omega,Y^*, Y^\dagger$ denote a triplet as in \eqref{eq:cb}, hence
\smash{$Y^* = Y + \sqrt\alpha \omega$} and 
\smash{$Y^\dagger = Y - \omega / \sqrt\alpha$}.  Consider    
$$
\E\big[\Var(\CB_\alpha(g) \,|\, Y)\big] = \frac{1}{B} \E\big[\Var\big(\|Y^\dagger
- g(Y^*)\|_2^2 - \|\omega\|_2^2/\alpha \,\big|\, Y \big)\big],
$$
where we used the independence of the bootstrap samples across $b=1,\ldots,B$.
We can therefore study the reducible variance for a single bootstrap draw, and
then for the final result, we simply need to divide by $B$.  To this end, let
\smash{$a = (2/\sqrt\alpha) \langle \omega, Y - g(Y^*) \rangle$}, $b = \|Y -  
g(Y^*)\|_2^2$, and write $\E_\omega, \Var_\omega, \Cov_\omega$ for the
expectation, variance, and covariance operators conditional on $Y$.  Then     
\begin{align*}
\Var\big(\|Y^\dagger - g(Y^*)\|_2^2 - \|\omega\|_2^2/\alpha \,\big|\, Y \big) 
&= \Var\big(\|Y^\dagger - Y + Y - g(Y^*)\|_2^2 - \|\omega\|_2^2/\alpha \,\big|\, 
Y \big) \\
&= \Var\big(\| Y - g(Y^*)\|_2^2 - (2/\sqrt\alpha) \langle \omega, Y - g(Y^*)
\rangle \,\big|\, Y \big) \\ 
&= \Var_\omega(a) + \Var_\omega(b) - 2\Cov_\omega(ab).
\end{align*}
The first term in the previous line $\Var_\omega(a)$ will end up having the
dominant dependence on $\alpha$, since by the law of total variance,
$$
\E[\Var_\omega(b)] = \E\big[ \Var_\omega\big(\|Y - g(Y^*)\|_2^2 \,\big|\, Y
\big)\big] \leq \Var(\|Y - g(Y^*)\|_2^2), 
$$
and the right-hand side above is continuous in $\alpha$ over $[0, \beta)$, by 
the condition $\E\|g(Y_\beta)\|_2^4 < \infty$ and Lemma 
\ref{lem:f_alpha_smoothness}, which means \smash{$\E[\Var_\omega(b)] \leq
  \Var(\|Y - g(Y)\|_2^2) + O(\alpha)$}.  Thus it remains to study
$\Var_\omega(a)$.  Introducing more notation, \smash{$c = (2/\sqrt\alpha)
  \langle \omega, Y - g(Y) \rangle$} and \smash{$d = (2/\sqrt\alpha)
  \langle \omega, g(Y) - g(Y^*) \rangle$}, observe that 
$$
\Var_\omega(a) = \Var_\omega(c) + \Var_\omega(d) + 2\Cov_\omega(cd).
$$
Once again, the first term here will have the dominant dependence on $\alpha$,
as 
$$
\Var_\omega(d) \leq \E_\omega[d^2] \leq \frac{4}{\alpha} n\sigma^2 
\E_\omega\|g(Y) - g(Y^*)\|_2^2,
$$
and the last factor on the right-hand side, after integrating over $Y$,
satisfies $\E\|g(Y) - g(Y^*)\|_2^2 = O(\alpha)$ from another application of
Lemma \ref{lem:f_alpha_smoothness}.  Finally,  
$$
\Var_\omega(c) = \frac{4n\sigma^2}{\alpha} \|Y - g(Y)\|_2^2,
$$
and integrating with respect to $Y$, then dividing by $B$, gives the desired
result in \eqref{eq:cb_rvar_bd}.

\subsection{Proof of Proposition \ref{prop:cb_ivar}}

Observe that \eqref{eq:cb_ivar} equals, for \smash{$a = [\E\|Y - g(Y +
  \sqrt\alpha \omega)\|_2^2]^2$} and \smash{$b = [\E\langle \omega, g(Y +
  \sqrt\alpha \omega) \rangle]^2 / \alpha$},   
$$
\int \bigg( \Big( \E\|y - g(y + \sqrt\alpha \omega)\|_2^2 + (2/\sqrt\alpha)
\E\big[\langle \omega, g(y + \sqrt\alpha \omega) \rangle\big] \Big)^2 - 
(a + b) \bigg) \frac{1}{(2\pi\sigma^2)^{n/2}} 
\exp\bigg\{\frac{-\|y - \theta\|^2}{2\sigma^2}\bigg\} \, dy.
$$
Abbreviating \smash{$\phi_{\theta, \sigma^2 I_n}(y)  = (2\pi\sigma^2)^{-n/2}
  \exp(-\|y - \theta\|^2/(2\sigma^2))$}, the integrand above is bounded by   
$$
2\E\|y - g(y + \sqrt\alpha \omega)\|_2^4 \,\phi_{\theta, \sigma^2 I_n}(y) +   
\frac{8}{\alpha}\E\big[\langle \omega, g(y + \sqrt\alpha \omega) \rangle \big]^2 
\phi_{\theta, \sigma^2 I_n}(y). 
$$
Note that the second term is dominated by \smash{$2H(y) \phi_{\theta, \sigma^2
    I_n}(y)$}, due to \eqref{eq:h_dominate}, which is integrable by assumption
($\E[H(Y)] < \infty$).  The first term above is dominated by    
$$
4\|y\|_2^2 \,\phi_{\theta, \sigma^2  I_n}(y) + 
4\E\|g(y + \sqrt\alpha \omega)\|_2^4 \,\phi_{\theta, \sigma^2  I_n}(y),  
$$
which is also integrable by assumption ($\E\|g(Y_\beta)\|_2^4 < \infty$).
Using Lebesgue's dominated convergence theorem (DCT) and \eqref{eq:h_limit}
completes the proof.      

\section{Additional experiments}
\label{app:additional_experiments}

\subsection{Bias}

We study the bias empirically, and investigate the tightness of the bound in 
\eqref{eq:cb_bias_bd2} in Proposition \ref{prop:cb_bias}.  Under the simulation 
setup described in Section \ref{sec:experiments}, with $s=5$ and
$\mathrm{SNR}=2$, Figure \ref{fig:cb_bias} displays the true bias (computed via
Monte Carlo) and \eqref{eq:cb_bias_bd2} each as functions of $\alpha$, when 
$g$ is forward stepwise regression estimator at different steps along its path:
$k=3$, 10, and 90.  We see that, within each panel, the bias decreases
approximately linearly with $\alpha$, meaning the linear rate of decay in the
bound \eqref{eq:cb_bias_bd2} is roughly accurate.  However, the slope in the
bound is too large, and loosest when $g$ is defined by the smallest number of
steps along the path.  This is consistent with the fact that bound
\eqref{eq:cb_bias_bd2} is based on applying the inequality 
\smash{$\Cor(\|\theta -  g(Y_t)\|_2^2, \|Y_t - \theta\|_2^2) \leq 1$} to the
integrand in \eqref{eq:cb_bias}.  This inequality is generally tightest when
$g(Y_t) = Y_t$, which occurs at $k=100$ steps (overfitting), and loosest at the
beginning of the path.

\begin{figure}[htb]
\centering
\includegraphics[width=\textwidth]{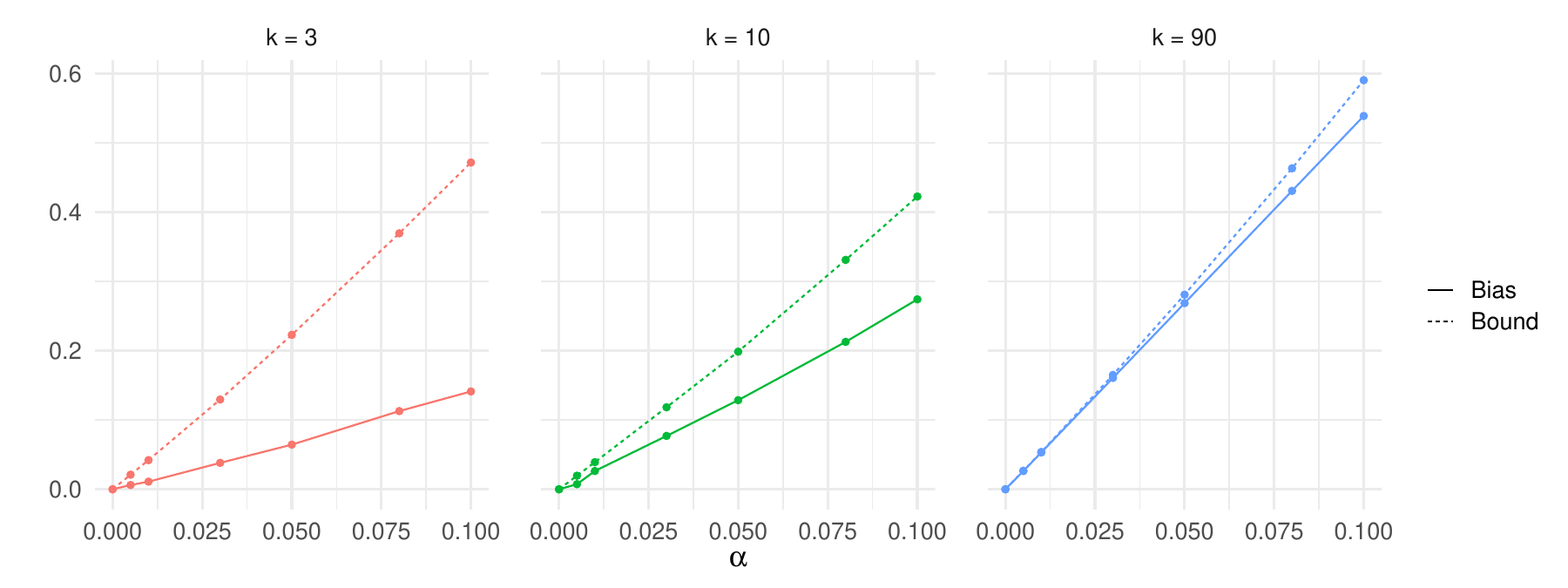}
\caption{Comparison of the true bias and the bound in \eqref{eq:cb_bias_bd2} for
  forward stepwise regression with $k=3$, 10, and 90 steps.  The simulation
  setup is as in Section \ref{sec:experiments} with $s=5$ and $\mathrm{SNR}=2$.} 
\label{fig:cb_bias}
\end{figure}

\subsection{Reducible variance} 

Now we examine the reducible variance empirically, and compare the bound in
\eqref{eq:cb_rvar_bd} in Proposition \ref{prop:cb_rvar}. We again use the
simulation setup from Section \ref{sec:experiments}, with $s=5$ and
$\mathrm{SNR}=2$, with Figure \ref{fig:cb_rvar} displays contour plots of the
true reducible variance (computed via Monte Carlo) and the dominant term in
\eqref{eq:cb_rvar_bd} as functions of $B$ and $\alpha$, when $g$ is the lasso
estimator with $\lambda = 0.31$.  The two panels appear qualitatively quite
similar, confirming that the dominant term in \eqref{eq:cb_rvar_bd} indeed  
captures the right dependence of the reducible variance on $B,\alpha$.  (Note
that each panel is given its own color scale, which means that any potential
looseness in the constant multiplying $1/(B\alpha)$ in the bound
\eqref{eq:cb_rvar_bd} is not being reflected.) 

\begin{figure}[htb]
\centering
\includegraphics[width=0.75\textwidth]{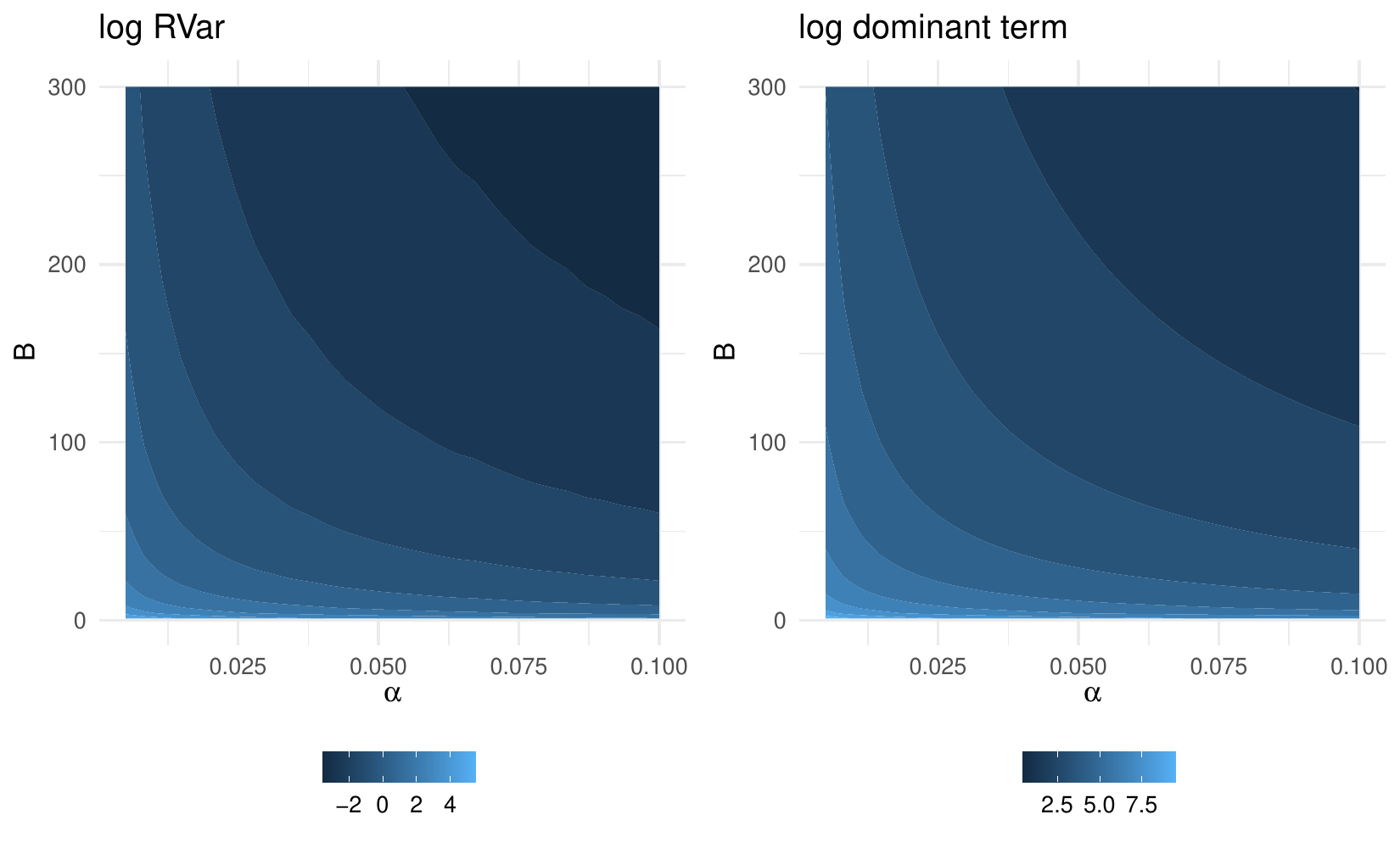}
\caption{Comparison of the true reducible variance and the bound in
  \eqref{eq:cb_rvar_bd} for the lasso with $\lambda=0.31$.  The simulation setup
  is as in Section \ref{sec:experiments} with $s=5$ and $\mathrm{SNR}=2$.}  
\label{fig:cb_rvar}
\end{figure}

\subsection{Irreducible variance} 

We study the behavior of the irreducible variance and its components
empirically.  Following \eqref{eq:cb_ivar}, observe that we can write    
\begin{multline*}
\IVar(\CB_\alpha(g)) = \underbrace{\Var\big( \E\big[ \|Y - g(Y + \sqrt\alpha
  \omega)\|_2^2 \,\big|\, Y \big] \big)}_{\IVar_1} + 
\underbrace{\Var\bigg( \frac{2}{\sqrt\alpha} \E\big[ \langle \omega, g(Y +
  \sqrt\alpha \omega) \rangle \,\big|\, Y \big] \bigg)}_{\IVar_2} +{} \\
\underbrace{2 \Cov\bigg( \E\big[ \|Y - g(Y + \sqrt\alpha 
  \omega)\|_2^2 \,\big|\, Y \big], \frac{2}{\sqrt\alpha} \E\big[ \langle \omega,
  g(Y + \sqrt\alpha \omega) \rangle \,\big|\, Y \big] \bigg)}_{\Cov_{1,2}}.
\end{multline*}
We can similarly define analogous components for $\IVar_1, \IVar_2, \Cov_{1,2}$
for $\IVar(\BY_\alpha(g))$ in \eqref{eq:by_ivar}.  Note that between the CB and
BY estimators, $\IVar_2$ is shared (equal), but $\IVar_1$ and $\Cov_{1,2}$ are
different: where BY uses the original training error $\|Y-g(Y)\|_2^2$, CB
substitutes the conditional expectation of the noise-added training error
\smash{$\E[ \|Y - g(Y + \sqrt\alpha \omega)\|_2^2 \,|\, Y]$}.  

Figure \ref{fig:cb_by_ivar} plots these three components of the irreducible
variance for BY and CB (computed via Monte Carlo), under the same
simulation setup as that from Figure \ref{fig:risk_comparison2}.  The figure
also plots the reducible variance for reference.  We can see that the main
contributor to the large variance exhibited by BY in comparison to CB in Figure
\ref{fig:risk_comparison2} is in fact the first component of the irreducible 
variance $\IVar_1$.  This is intuitive, because for an unstable function $g$
(such as the one in the current simulation), the observed training error can
have a high degree of variability, but taking a conditional expectation over a
noise-adding process acts as a kind of regularization, reducing this variability
greatly. 

\begin{figure}[htb]
\centering
\includegraphics[width=\textwidth]{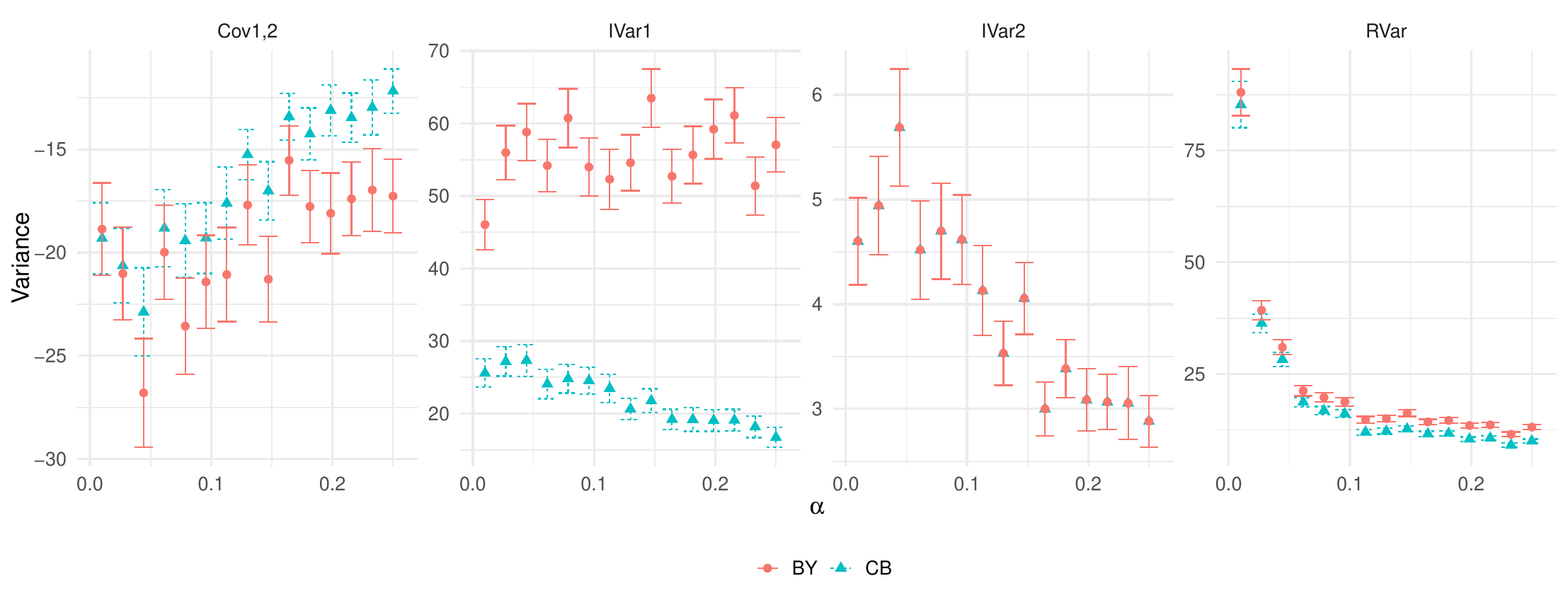}
\caption{Comparison of the irreducible variance, broken down into its three main 
  components, and also the reducible varaince, for the BY and CB estimators,
  under the same simulation setup as that in Figure \ref{fig:risk_comparison2}.}
\label{fig:cb_by_ivar}
\end{figure}

\subsection{Model selection}

Below we report the results of repeating the experiment in Section
\ref{sec:image_denoising} over 10 draws of synthetic noise, to create 10 noisy  
images. For each such noisy image, we compute the CB and SURE risk curves and 
calculate the minimizing values of $\lambda$. Figure \ref{fig:hist_lambdas}
displays a histogram of these selected $\lambda$ values, from each method. We
can see that CB often selects $\lambda$ values which are tightly coupled around
those selected by SURE, suggesting some degree of stability in model selection.      

\begin{figure}[htb]
\centering
\includegraphics[width=0.7\textwidth]{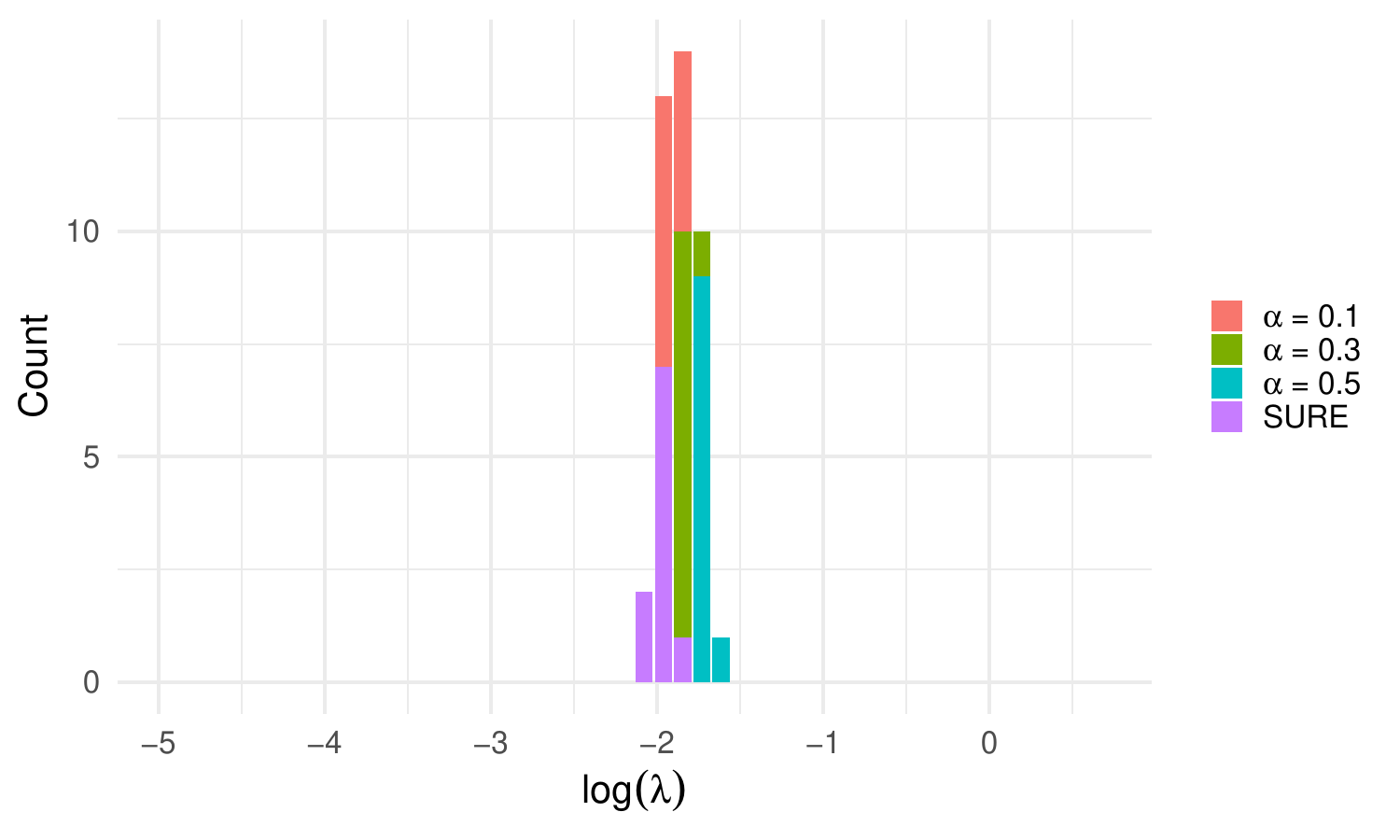}
\caption{Histogram of the minimizing values of $\lambda$ for the CB (at $\alpha
  \in \{0.1, 0.3, 0.5\}$) and SURE curves, over 10 repetitions of the simulation
  setup that generated Figure \ref{fig:parrot_risk}.} 
\label{fig:hist_lambdas}
\end{figure}

\end{document}